\documentclass[10pt]{amsart}
\usepackage{newlfont}
\usepackage{epsfig}
\usepackage{mathrsfs}
\usepackage{color}
\usepackage{xcolor}
\usepackage{amsthm}
\usepackage{amsmath}
\usepackage{amsfonts}
\usepackage{amssymb}

\usepackage{graphicx}
\newtheorem{corollary}{Corollary}[section]

\newtheorem{lemma}[corollary]{Lemma}

\newtheorem{remark}[corollary]{Remark}
\newtheorem{theorem}[corollary]{Theorem}

\newfont{\sBlackboard}{msbm10 scaled 900}

\newcommand{\mylabel}[1]{\label{#1}
            \ifx\undefined\stillediting
            \else \fbox{$#1$}\fi }
\newcommand{\BE}{\begin{equation}}

\newcommand{\EEQ}{\end{equation}}
\newcommand{\rfb}[1]{\mbox{\rm
   (\ref{#1})}\ifx\undefined\stillediting\else:\fbox{$#1$}\fi}

\newfont{\Blackboard}{msbm10 scaled 1200}

\newfont{\roma}{cmr10 scaled 1200}

\def\CC{\rm \hbox{C\kern-.56em\raise.4ex
         \hbox{$\scriptscriptstyle |$}\kern+0.5 em }}

\def\L{\Lambda}

\def\R{\mathbb{R}}
\def\N{\mathbb{N}}

\def\LG{{\textbf {\textit L}}}

\newcommand{\mm}    {{\hbox{\hskip 0.5pt}}}

\newcommand{\bluff} {{\hbox{\raise 15pt \hbox{\mm}}}}

\renewcommand{\L}    {{\Lambda}}

\makeatletter
\def\section{\@startsection {section}{1}{\z@}{-3.5ex plus -1ex minus
    -.2ex}{2.3ex plus .2ex}{\large\bf}}
\makeatother
\def\be{\begin{equation}}
\def\ee{\end{equation}}

\def\ds{\displaystyle}

\begin{document} 
\thispagestyle{empty} 
\title{Resolvent estimates for wave operators in Lipschitz domains}
\author{Ka\"is AMMARI}
\address{UR Analysis and Control of PDEs, UR 13E64, Department of Mathematics,  
Faculty of Sciences of Monastir, University of Monastir, 5019 Monastir, Tunisia}
\email{kais.ammari@fsm.rnu.tn}  
\author{Ch\'erif AMROUCHE}
\address{Laboratoire de Math\'ematiques et de leurs applications,
UMR CNRS 5142
B\^atiment IPRA - Universit\'e de Pau et des Pays de l'Adour
Avenue de l'Universit\'e - BP 1155
64013 Pau Cedex, France}
\email{cherif.amrouche(@)univ-pau.fr}

\begin{abstract}
In this paper we study the resolvent of wave operators on open and bounded Lipschitz domains of $\mathbb{R}^N$ with Dirichlet or Neumann boundary conditions. We give results on existence and estimates of the resolvent for the real and complex cases. 
\end{abstract}


\maketitle

\tableofcontents


\section{Introduction}  \label{intro}

Let $\Omega \subset \mathbb{R}^N,$ with $ N \geq2,$ be a non-empty bounded open Lipschitz domain with boundary $\Gamma$. We consider here the following equations: 
\begin{equation}\label{PD}
(\mathcal{P}_D)\ \ \ \  \lambda^2 u - \Delta u = f\quad \ \mbox{in}\ \Omega \quad
\mbox{and } \quad u = g \ \ \mbox{on }\Gamma,
\end{equation}
and
\begin{equation}\label{PN}
(\mathcal{P}_N)\ \ \ \ \lambda^2 u - \Delta u = f\quad \ \mbox{in}\ \Omega \quad
\mbox{and } \quad \frac{\partial u}{\partial\textit{\textbf n}} = h \ \ \mbox{on }\Gamma,
\end{equation}
where $\lambda \in \mathbb{C}$, $\textit{\textbf n}$ is the unit normal of $\Gamma$ pointing towards the exterior of $\Omega$, with data in some Sobolev spaces. When $g = 0$ (respectively $h = 0$), we denote this problem by $(\mathcal{P}_D^0)$ (respectively $(\mathcal{P}_N^0))$ but when $f = 0$, we denote this problem by $(\mathcal{P}_D^H)$ (respectively $(\mathcal{P}_N^H))$. 

\medskip

These equations can be interpreted as a wave equations (wave equation in the frequency domain) and appear naturally from general conservation laws of physics. The wave equation can also be derived from the telegraph equation, problems associated with steady-state oscillations (mechanical, acoustic, electromagnetic, etc.), and the problems of diffusion of some gases in the presence of decay or chain reactions and other wave-type, or evolutionary equations. 

\medskip

In this paper, we study the existence, the regularity of the solutions of (\ref{PD}) and (\ref{PN}) in function of $\lambda$ real or complex (Re $\lambda > 0$) and when $\Omega$ is an open bounded Lipschitz domain. We note, in particular, that in this paper
we consider weaker hypotheses than those considered by Jerison and Kenig for the Poisson problem on Lipschitz domain \cite{Jer1, Jer2, J-K} ($f \in L^2 (\Omega)$), for the existence of a solution $u \in H^{3/2}(\Omega)$ and we obtain optimal estimates in $\lambda$. 

The originality of our work lies in the optimal estimations of the resolvent, which we do not find, to our knowledge, anywhere else.

\medskip

We note that we can obtain analogues results for Helmholtz equation with real frequency, {\it i.e} $\lambda$ is purely complex. Indeed, by applying the results of the present paper with a complex frequency, that is an order-one perturbation of the real frequency, and then solving away the error from this perturbation by using bounds on the $L^2$ norm of the solution in terms of the data, as obtained for scattering problems, by Morawetz \cite{morawetz}, Vainberg \cite{vainberg}, Melrose-Sj\"ostrand \cite{MS} (see also \cite{MS} and references therein), Morawetz-Ralson-Strauss \cite{MRS}.

\medskip

The paper is organized as follows: in Section \ref{sec2}, we give some results of existence, regularity and wavenumber-explicit bounds of corresponding maps associated to the Laplace ($\lambda = 0$) and the resolvent of the wave operators ($\lambda\in \mathbb{R}^*$). In Section \ref{sec3}, we treat (existence, regularity and wavenumber-explicit dependance) the complex case ($\lambda \in \mathbb{C}$ and Re $\lambda > 0$).

\section{The Resolvent estimates. Real Case} \label{sec2}
\setcounter{equation}{0}

In this section we study the existence and regularity of solutions of $(\mathcal{P}_D), (\mathcal{P}_D^H), (\mathcal{P}_N)$ and 
$(\mathcal{P}_N^H)$  for $\lambda \in \mathbb{R}.$

\subsection{Case of the Laplace equation: $\lambda = 0$}
 
These problems when $\lambda = 0$ were studied by many authors (see \cite{Jer1, Jer2, J-K, Necas}) and now  we  start  by recalling some results here. The first result is devoted to the famous Ne\v{c}as property, which plays an important role in the analysis of elliptic problems and in particular in the study of the Steklov-Poincar\'e operator. 

\begin{theorem}\label{Necas Property}$\mathrm{(}${\bf Ne\v{c}as Property}$\mathrm {)}$ Let 
$$
u\in H^1(\Omega) \quad with \quad   \Delta u\in L^2(\Omega).
$$
i) Then we have the following property: 
\begin{equation}\label{Necasproperty}
u_{\vert\Gamma} \in H^1(\Gamma) \quad \Longleftrightarrow\quad \frac{\partial u}{\partial \textbf{\textit n}}\in L^2(\Gamma) \quad \Longleftrightarrow \quad \nabla u \in (L^2(\Gamma))^N,
\end{equation}
with the following estimates:
\begin{equation*}\label{d02-100118-e2}
\ds\left|\left|\frac{\partial u}{\partial \textit{\textbf n}} \right|\right|_{ \textit{L}^2(\Gamma)}\ \leq\ C\left(\inf_{k\in \R}\left|\left |u + k \right |\right |_{H^1(\Gamma)} +\left|\left| \Delta u \right|\right|_{L^2(\Omega)}\right)
\end{equation*}
and 
\begin{equation*}\label{d02-100118-e3}
\ds \inf_{k\in \R}\left|\left|u + k \right|\right|_{H^1(\Gamma)}  \ \leq\ C\left( \left|\left|\frac{\partial u}{\partial \textit{\textbf n}} \right|\right|_{ \textit{ L}^2(\Gamma)} +\left|\left| \Delta u \right|\right|_{L^2(\Omega)}\right).
\end{equation*}
ii) Moreover if one condition in \eqref{Necasproperty} is satisfied, then $u\in H^{3/2}(\Omega)$ and $\sqrt \rho\,\nabla^2 u \in (L^2(\Omega))^{N\times N}$, where $d$ is the distance to $\Gamma$.\\ 
iii) When $u$ is harmonic, the Steklov-Poincar\'e operator $S: u\mapsto  \frac{\partial u}{\partial \textbf{\textit n}}$ satisfies:
\begin{equation*}\label{d07-150318-e1}
S:\  H^1(\Gamma)\  \longrightarrow\  L^2(\Gamma)  \quad \mathrm{and}\quad S:\ H^1(\Gamma)/\R\  \longrightarrow\  L^2_0(\Gamma) 
\end{equation*}
are continuous. The second operator above is one to one and onto. Here  $L^2_0(\Gamma)$ denotes the space $ \left\{ \varphi \in L^2(\Gamma); \, \int_\Gamma \varphi   = 0\right\}$.\\
iv) Since $S$ is self adjoint, we get by duality the following properties: the following operators
\begin{equation*}\label{d07-150318-e1dual}
S:\  L^2(\Gamma)\  \longrightarrow\  H^{-1}(\Gamma)   \quad \mathrm{and}\quad S:\ L^2(\Gamma)/\R\  \longrightarrow\   H^{-1}(\Gamma) \, \perp \R, 
\end{equation*}
where  $H^{-1}(\Gamma) \, \perp \R =  \left\{ f \in H^{-1}(\Gamma) ; \, \langle f,\, 1\rangle  = 0\right\}$, are continuous. The second operator above is one to one and onto. 
\end{theorem}
\medskip

The following theorem concerns the existence with optimal regularity of solutions to Problem $(\mathcal{P}_D)$, with $\lambda = 0$, see  \cite{Jer1} \cite{Jer2} \cite{J-K}.

\begin{theorem}\label{PDH, lambda=0}$\mathrm{(}${\bf Existence in} $H^s(\Omega)$  {\bf  for} $(\mathcal{P}_D)$  {\bf with} $\lambda=0 \mathrm {)}$ Let $\frac{1}{2}\leq s\leq\frac{3}{2}$ and  $\lambda = 0$. 

\noindent i) Then for any $g\in H^{s-1/2}(\Gamma)$, Problem $(\mathcal{P}_D^H)$ has a unique solution $u\in H^s(\Omega)$ with the estimate 
\begin{equation*}\label{d02-240118-e1c}
\left|\left| u\right|\right|_{H^s(\Omega)}\ \leq \ C\,\left|\left| g\right|\right|_{H^{s-1/2}(\Gamma)}.
\end{equation*}
Moreover,  if $s=1/2$, then $\sqrt \rho\,\nabla u \in (L^2(\Omega))^N,$ with the estimate
\begin{equation*}\label{d02-240118-e1cche1}
\left|\left| u\right|\right|_{H^{1/2}(\Omega)} + \left\Vert \sqrt \rho \,\nabla u \right\Vert_{\textit{\textbf L}^2(\Omega)}\ \leq \ C\,\left|\left| g\right|\right|_{L^{2}(\Gamma)}
\end{equation*}
and if $s=3/2$, then $\sqrt \rho\,\nabla^2 u \in (L^2(\Omega))^{N\times N},$ with the estimate 
\begin{equation*}\label{d02-240118-e1cex}
\left|\left| u\right|\right|_{H^{3/2}(\Omega)}+\left|\left| \sqrt \rho\,\nabla^2 u\right|\right|_{\textit{\textbf L}^2(\Omega)} \ \leq \ C\,\left|\left| g\right|\right|_{H^{1}(\Gamma)}.
\end{equation*} 
Here $\rho$ denotes the distance function to the boundary: for $x\in \Omega$, $\rho(x) = d(x; \, \Gamma)$.\\
ii) For any $f\in H^{s-2}(\Omega)$, with $\frac{1}{2} < s < \frac{3}{2}$,  Problem $(\mathcal{P}_D^0)$ has a unique solution $u\in H^s(\Omega)$ with the estimate 
\begin{equation*}\label{d02-240118-e1cb}
\left|\left| u\right|\right|_{H^s(\Omega)}\ \leq \ C\,\left|\left| f\right|\right|_{H^{s-2}(\Omega)}.
\end{equation*}
iii) If $f\in L^2(\Omega)$, then $u\in H^{3/2}(\Omega)$.
\end{theorem}

\begin{remark}\upshape This last result in Point iii) is of course not optimal. So, we will see below how to choose the RHS $f$ in order to get the solution in $H^{3/2}(\Omega)$.
\end{remark}

For that, instead to use classical Sobolev spaces, we need to consider adapted weighted Sobolev spaces. Let us introduce the following space: For $m\in \N$ and $r\in \mathbb{R}$
$$
 \ds \mathscr{Q}^m_{-r}(\Omega)\ =\ \left\{ v\in \mathcal{D}'(\Omega);\ \ \rho^{|\lambda|-r}D^\lambda v\in \LG^2(\Omega), \;    \vert\lambda\vert \leq   m\right\},
$$
which is a Hilbert space for the norm
$$
 ||u||_{\mathscr{Q}^m_{-r}(\Omega)}\ =\ \left(\sum_{ |\lambda|\leq m} ||  \rho^{|\lambda|-r}D^\lambda v ||^2_{ \LG^2(\Omega)}\right)^{1/2}.
$$
The case $r = 0$ is treated in Lions-Magenes, Definition 6.1-Chap.II. When $m = 0$, we sometimes denote the corresponding space as follows:
$$
 \mathscr{Q}^0_{-r}(\Omega) = L^2_{\rho^{-r}}(\Omega).
 $$

We now define the following space: For a real $s$ such that $s=m+\theta$ with $m\in\N$, $r\in \mathbb{R}$ and $0<\theta<1$, we set: 
$$
\mathscr{Q}^s_{-r}(\Omega)\ =\ \left[\mathscr{Q}^{m+1}_{-r}(\Omega),\mathscr{Q}^m_{-r}(\Omega)\right]_{1-\theta}.
$$
As $\mathcal{D}(\Omega)$ is dense in the space $\mathscr{Q}^m_{-r}(\Omega)$, we deduce that for any real $ r$  and for any real  $s\geq 0$,
\begin{equation}\label{a26-e1}
\mathcal{D}(\Omega)\ \textrm{ is dense in  } \mathscr{Q}^s_{-r}(\Omega).
\end{equation}
Therefore $\mathscr{Q}^s_{-r}(\Omega)$ is a normal space of distriburtions on $\Omega$ and its dual space denoted by 
$\mathscr{Q}^{-s}_{r}(\Omega)$  may be identified to a subspace of distributions on $\Omega$.
It may be represented  in the following form, when $s = m$ is an integer:
\begin{equation}\label{a26-e3}
 \mathscr{Q}^{-m}_{r}(\Omega)\ =\ \left\{f=\sum_{|\lambda|\leq m}D^\lambda\left(\rho^{|\lambda|-r}f_\lambda \right);\ \ f_\lambda\in L^2(\Omega) \right\}.
\end{equation}

\begin{remark}\upshape
Recall that for any non negative integer $m$, we have $\mathscr{Q}^m_{-m}(\Omega) = H^m_0(\Omega)$, since
$$
\ds v\in H^m_0(\Omega)\ \Longleftrightarrow \rho^{\vert\lambda\vert - m} D^{\vert\lambda\vert}v\in \LG^2(\Omega) \quad\mathrm{for \, all }\quad \vert\lambda\vert \leq m.
$$
In particular for $m = 2$, we have the following equivalence norms:
for any  $v\in H^2_{0}(\Omega) $, 
\begin{equation}\label{equivnormH20}
\Vert  v  \Vert_{H^2_{0}(\Omega) } \simeq \Vert  v  \Vert_{\mathscr{Q}^2_{-2}(\Omega) }\simeq \Vert\Delta v  \Vert_{L^2(\Omega)}.
\end{equation}
For $m = 1$ and  $v\in \mathscr{Q}^2_{-1}(\Omega)$, we have the following estimate
\begin{equation}\label{inegrhogradient2}
\Vert  v  \Vert_{\mathscr{Q}^2_{-1}(\Omega)} \leq  C \left(\Vert \rho \Delta v  \Vert_{L^2(\Omega)} + \Vert \frac{v}{\rho} \Vert_{L^2(\Omega)} \right).
\end{equation}
\end{remark}

Before stating the first result on weighted spaces, we recall some specific properties of Lipschitz domains in $\mathbb{R}^N$ needed in the sequel. The first one is given in Theorem 2.2 of \cite{ACM} (see also \cite{BZ} and \cite{BF}): there exist subdomains $\Omega_k$, $k\ge 1$ such that
$$
 \Omega = \bigcup_{k\geq 1}\Omega_k,\quad \mathrm{with}\quad \Omega_k \; \mathrm{of\, \, class} \; \; \mathcal{C}^{\infty},\quad\overline{ \Omega_k} \subset  \Omega_{k+1}
$$
and where 

$$
\Omega_k  = \left\{x\in \Omega;\; \rho^\star(x)  > \frac{1}{k}\right\}.
$$
The function $\rho^\star$ is the regularized signed distance to $\Gamma$, which satisfies:
\begin{equation}\label{rostar}
\forall x\in \Omega, \quad C_1 \rho (x) \leq \rho^\star (x) \leq C_2 \rho (x).
\end{equation}
Observe that the distance function $\rho_k$ to the boundary $\Gamma_k$ of $\Omega_k$ is given here by $\rho^\star - 1/k$. The second one is given in  Lemma 3.1, Chapter 6 in \cite{Necas} :  there exists a function $\sigma$ belonging to $\mathcal{C}^\infty(\Omega) \cap \mathcal{C}^{0,1}(\overline{\Omega})$ and such that for any $x\in \Omega $
\begin{equation}\label{sigma}
C_1 \rho(x) \leq \sigma (x) \leq C_2 \rho (x) \quad \mathrm{and}\quad \vert D^\lambda \sigma \vert \leq C \sigma^{1 - \vert  \lambda\vert} \quad \mathrm{for\, any\, multi}-\mathrm{index}\; \lambda.
\end{equation}
In the following we will use one or other of the functions $\rho, \rho^\star$ or $\sigma$ alternately, depending on our needs.

\begin{theorem}\label{isoL^2rho}
For any $f\in L^2_{\rho}(\Omega)$, there exists a unique solution $u\in \mathscr{Q}^2_{-1}(\Omega)$ satisfying $\Delta u = f$ in $\Omega$. Moreover
\begin{equation}
\Vert  u  \Vert_{\mathscr{Q}^2_{-1}(\Omega)} \leq  C\Vert \rho f  \Vert_{L^2(\Omega)}. 
\end{equation}
\end{theorem}

\begin{proof}
Since $L^2_{\rho}(\Omega)$ is included in $H^{-1}(\Omega)$, we know that there exists a unique solution $u\in H^1_0(\Omega)$ satisfying $\Delta u = f$ in $\Omega$ and
\begin{equation}
\Vert  u  \Vert_{H^1_0(\Omega)}\leq C \Vert  f  \Vert_{H^{-1}(\Omega)}  \leq  C\Vert \rho f  \Vert_{L^2(\Omega)}. 
\end{equation}
Writing now
$$
\Delta \left(\sigma\frac{\partial u}{\partial x_j}\right) = \sigma \frac{\partial f}{\partial x_j} + 2\nabla \sigma \cdot \nabla \frac{\partial u}{\partial x_j} + \frac{\partial u}{\partial x_j}\Delta \sigma,
$$
we observe that each term in the RHS belongs to $H^{-1}(\Omega)$. For instance, we have for any $\varphi \in \mathcal{D}(\Omega)$
$$
\left| \langle \frac{\partial u}{\partial x_j}\Delta \sigma, \, \varphi\rangle\right| = \left| \int_\Omega \frac{\partial u}{\partial x_j} \varphi\Delta \sigma \right| \leq C \int_\Omega \left|\frac{\partial u}{\partial x_j} \right|  \left| \frac{\varphi}{\sigma}\right| \leq C \Vert \rho f  \Vert_{L^2(\Omega)} \left\|\frac{\varphi}{\rho}\right\|_{L^2(\Omega)},
$$
where the constant $C$ depends only on the Lipschitz character of $\Omega$.
At this point, since $\Omega$ is only Lipschitz and $\sigma\frac{\partial u}{\partial x_j}$ belongs only to $L^2(\Omega)$, we are not able to deduce that $\sigma\frac{\partial u}{\partial x_j}$ is also an element of $H^1_0(\Omega)$ (see Remark \ref{L2pasH1} below). To get around this difficulty, we will approach the domain $\Omega$ by taking the sequence of regular open sets $\Omega_k$ (see above).  Let $u_k\in H^2(\Omega_k)\cap H^1_0(\Omega_k)$ be the unique solution satisfying $\Delta u_k = f$ in $\Omega_k$ with the estimate
\begin{equation}
\Vert\nabla u_k\Vert_{L^2(\Omega_{k})} \leq C \Vert \rho_k f \Vert_{ L^2(\Omega_{k})} \leq C \Vert \rho f\Vert_{ L^2(\Omega)},
\end{equation}
where the constant $C$ depends only on the Lipschitz character of $\Omega$. Setting now $\widetilde{u_k} = u_k$ in $\Omega_k$ and  $\widetilde{u_k} = 0$ in $\Omega\setminus \Omega_k$, we deduce that the sequence $(\widetilde{u_k})_k$ is bounded in $H^1_0(\Omega).$  So $\widetilde{u_k} \rightharpoonup u^*$ in $H^1_0(\Omega)$. We shall prove that $u^* = u$. Indeed, let $\varphi \in \mathcal{D}(\Omega)$. Then supp $\varphi \subset \Omega_{k_{0}}$ for $k_0$ sufficiently large and for any $k \geq k_0$, we have
$$
\langle \Delta \widetilde{u_k}, \, \varphi\rangle = - \int_\Omega \nabla \widetilde{u_k}\cdot \nabla \varphi = \int_\Omega f\, \varphi .
$$ 
Hence 
$$
\lim_{k\rightarrow \infty} \langle \Delta \widetilde{u_k}, \, \varphi\rangle =  -  \int_\Omega \nabla u^*\cdot \nabla \varphi = \langle \Delta u^*, \, \varphi\rangle .
$$
As a consequence $\Delta u^* = f$ in $\Omega$ and then $u^* = u.$

Besides, using the estimate \eqref{inegrhogradient2}, we have 
\begin{equation}\label{estimrhograd2}
\Vert 1_{\Omega_k}\rho_k\nabla^2\widetilde{u_k})_k\Vert_{L^2(\Omega)}\leq C \Vert \rho f\Vert_{ L^2(\Omega)}.
\end{equation} 
So that for any $1\leq i, j \leq N$, we can extract a subsequence, denoted by the same way, such that
\begin{equation}\label{weakcgce}
1_{\Omega_k}\rho_k\frac{\partial^2 \widetilde{u_k}}{\partial x_i\partial x_j} \rightharpoonup z_{ij} \quad \mathrm{in}\quad L^2(\Omega).
\end{equation}
As previously, let $\varphi \in \mathcal{D}(\Omega)$ such that supp $\varphi \subset \Omega_{k_{0}}$ for $k_0$ sufficiently large. Then for any $k \geq k_0$,
\begin{equation*}
\int_\Omega 1_{\Omega_k}\rho_k\frac{\partial^2 \widetilde{u_k}}{\partial x_i\partial x_j}\, \varphi = - \int_{\Omega} \frac{\partial u_k}{\partial x_j}\, \frac{\partial}{\partial x_i}(\rho_k\varphi )\longrightarrow  - \int_\Omega \frac{\partial u}{\partial x_j}\, \frac{\partial}{\partial x_i}(\rho^\star\varphi ) = \int_\Omega \rho^\star\frac{\partial^2 u}{\partial x_i\partial x_j}\, \varphi.
\end{equation*}

That means that $ z_{ij} = \rho^\star\frac{\partial^2 u}{\partial x_i\partial x_j}$ and the desired estimate thanks to  \eqref{rostar}, \eqref{estimrhograd2} and  \eqref{weakcgce}.
\end{proof} 

\begin{remark} \label{L2pasH1} \upshape Suppose that $v\in L^2(\Omega)$ is a function which satisfies $\Delta v \in H^{-1}(\Omega)$. When the bounded domain $\Omega$ is of class $\mathcal{C}^{1,1}$, we know that $v$ has a trace in $H^{-1/2}(\Gamma)$ (see \cite{ARB}). Moreover if $v_{\vert\Gamma}\in H^{1/2}(\Gamma)$, then $v\in H^{1}(\Omega)$. In the case where $\Omega$ is only Lipschitz, this regularity result does not hold, as we can see in the following counter example.
Indeed, let us consider the following Lipschitz domain in the case $N = 2$: for $\frac{1} {2} < \alpha < 1$, 
$$
\Omega = \{(r, \theta);\; 0 < r < 1,\quad 0 < \theta < \frac {\pi} {\alpha}\}
$$
We can easily verify that the following function
$$
v(r, \theta) = (r^{-\alpha} - r^{\alpha})\mathrm{sin}(\alpha\theta)
$$
is harmonic in $\Omega$ with $v = 0$ on $\Gamma$ and $v\in W^{1, p}(\Omega)\cap L^2(\Omega)$ for any $p < \frac{2}{\alpha + 1}$. However $v\notin H^1(\Omega)$.

\end{remark}

\begin{corollary}\label{isoQ} For any $0 \leq \theta \leq 1$ and $v\in \mathscr{Q}^2_{-1-\theta}(\Omega)$,
\begin{equation}\label{inegrhothetagradient2}
\Vert  v  \Vert_{ \mathscr{Q}^2_{-1-\theta}(\Omega)} \leq  C\Vert \rho^{1-\theta} \Delta v  \Vert_{L^2(\Omega)}. 
\end{equation}
\end{corollary}

\begin{proof} From Theorem \ref{isoL^2rho}, we know on one hand that for any $v\in  \mathscr{Q}^2_{-1}(\Omega)$ and in particular for any $v\in \mathcal{D}(\Omega)$
\begin{equation}\label{inegrho1gradient2b}
\Vert  v \Vert_{L^2_{1/\rho}(\Omega)} + \Vert \nabla v \Vert_{L^2(\Omega)} + \left\| \frac{\partial^2 v}{\partial x_i\partial x_j}\  \right\|_{L^2_{\rho}(\Omega)} \leq  C\Vert  \Delta v  \Vert_{L^2_{\rho}(\Omega)}. 
\end{equation}
On the other hand for such a $v$, 
\begin{equation}\label{inegrho1gradient2b1}
\Vert  v \Vert_{L^2_{1/{\rho^2}}(\Omega)} + \Vert \nabla v \Vert_{L^2_{1/\rho}(\Omega)} + \left\| \frac{\partial^2 v}{\partial x_i\partial x_j}\  \right\|_{L^2(\Omega)} \leq  C\Vert  \Delta v  \Vert_{L^2(\Omega)}. 
\end{equation}
We deduce by interpolation that for any $v\in \mathcal{D}(\Omega)$ and for any $0 < \theta < 1$
\begin{equation}\label{inegrho1gradient2b2}
\Vert  v \Vert_{L^2_{1/{\rho^{1+\theta}}}(\Omega)} + \Vert \nabla v \Vert_{L^2_{1/\rho^{\theta}}(\Omega)} + \left\| \frac{\partial^2 v}{\partial x_i\partial x_j}\  \right\|_{L^2_{\rho^{1-\theta}}(\Omega)} \leq  C\Vert  \Delta v  \Vert_{L^2_{\rho^{1-\theta}}(\Omega)}. 
\end{equation}
and finally the required estimate in $ \mathscr{Q}^2_{-1-\theta}(\Omega)$ since $ \mathcal{D}(\Omega) $ is dense in $ \mathscr{Q}^2_{-1-\theta}(\Omega)$.

\end{proof}

\begin{remark}\upshape
In the particular case where $\theta = 1/2$,  we get from the above corollary the following estimate:
$$
\forall v\in \mathscr{Q}^2_{-3/2}(\Omega), \quad \Vert v \Vert_{H^{3/2}_{00}(\Omega)} + \Vert \rho^{1/2} \nabla^2 v   \Vert_{L^2(\Omega)} \leq  C\Vert \rho^{1/2} \Delta v  \Vert_{L^2(\Omega)}.
$$
So, since $\mathcal{D}(\Omega)$ is dense in the space denoted by $ \mathscr{T}^2_{-3/2}(\Omega)$, defined as follows
\begin{equation}
\mathscr{T}^2_{-3/2}(\Omega) =  \left\{ v\in H^{3/2}_0(\Omega);\;  \rho^{1/2}\nabla^2 v \in \ \LG^2(\Omega)\right\},
\end{equation}
which is an Hilbert space for his graph norm, we have obviously the following property: 
\begin{equation}\label{inegT2-3/2}
\forall v\in \mathscr{T}^2_{-3/2}(\Omega), \quad \Vert v \Vert_{H^{3/2}_{0}(\Omega)} + \Vert \rho^{1/2} \nabla^2 v   \Vert_{L^2(\Omega)} \leq  C\Vert \rho^{1/2} \Delta v  \Vert_{L^2(\Omega)}.
\end{equation}
Observe that $\mathscr{Q}^2_{-3/2}(\Omega) $ is a proper subspace of $ \mathscr{T}^2_{-3/2}(\Omega) $ with a finer topology. 
\end{remark}

\begin{theorem}\label{isodansEtheta} Let $f\in L^2_{\rho^{\theta}}(\Omega)$, with $1/2\leq \theta \leq 1$. Then, there exists a unique solution $u\in \mathscr{Q}^2_{- 2 +\theta}(\Omega) $ if $\theta\not= 1/2$, $\mathrm{(}$resp. $u\in  \mathscr{T}^2_{-3/2}(\Omega) $ if $\theta = 1/2\mathrm{)}$ satisfying $\Delta u = f$ in $\Omega$, with the corresponding estimate.
\end{theorem}

\begin{proof} We give a short proof  which takes the same ideas and the same notations of that of Theorem \ref{isoL^2rho}. We will prove only the case $\theta = 1/2$ since the proof is very similar for the case $\theta\not= 1/2$ . As $L^2_{\rho^{1/2}}(\Omega)$ is included in $H^{s-2}(\Omega)$ for any $1 \leq s < 3/2,$ we know that there exists a unique solution $u\in H^{s}_0(\Omega)$ satisfying $\Delta u = f$ in $\Omega$ for any $1 \leq s < 3/2.$ In order to prove that $u\in  \mathscr{T}^2_{-3/2}(\Omega) $, we take again as above the sequence of regular subdomains $\Omega_k$. 
We know that the mapping
$$
\Delta:\ \ \ H^{2}(\Omega_k)\cap H^1_0(\Omega_k)\ \ \ \longrightarrow\ \ \ L^2(\Omega_k)
$$
is then an isomorphism. As the following mapping
$$
\Delta:\ \ \ \mathscr{Q}^{2}_{-1}(\Omega_k)\ \ \ \longrightarrow\ \ \ L^2_{\rho_{k}}(\Omega_k)
$$
is also an isomorphism. Using \eqref {inegT2-3/2} we conclude that
\begin{equation}\label{a28-e2}
\Delta:\ \ \ \mathscr{T}^2_{-3/2}(\Omega_k) \ \ \ \longrightarrow\ \ \ L^2_{\sqrt\rho_k}(\Omega_k)
\end{equation}
is also an isomorphism. As a consequence, we get the following characterization
\begin{equation}\label{a28-e1}
\ds\left[H^2(\Omega_k)\cap H^1_0(\Omega_k), \mathscr{Q}^{2}_{-1}(\Omega_k)\right]_{1/2}\ =\ \mathscr{T}^2_{-3/2}(\Omega_k).
\end{equation}

So let $u_k\in H^{3/2}_0(\Omega_k)\cap H^2(\Omega_k)\subset \mathscr{T}^2_{-3/2}(\Omega_k)$ be the unique solution satisfying $\Delta u_k = f$ in $\Omega_k$. From the inequality \eqref{inegT2-3/2}, we get
\begin{equation}
\Vert u_k\Vert_{\mathscr{T}^2_{-3/2}(\Omega_k)} \leq C \Vert \sqrt{\rho_k} f \Vert_{ L^2(\Omega_{k})} \leq C \Vert \sqrt{\rho} f\Vert_{ L^2(\Omega)},
\end{equation}
where the constant $C$ depends only on the Lipschitz character of $\Omega$. The rest of the proof is very similar to that of Theorem \ref{isoL^2rho}.
\end{proof}

 
\begin{remark}\upshape\label{d06-140318-p1} Is easy to prove the previous result in the case where $\Omega$ is a bounded open subset of $\R^N$ of class $C^{1,1}$. Since the regularity of $\Omega$, for example the following mapping, corresponding to $\theta = 1/2$
\begin{equation}\label{a28-e2bis}
\Delta:\ \ \ \left\{ v\in H^{3/2}(\Omega)\cap H^1_0(\Omega);\ \ \ \sqrt\rho D^2v\in \LG^2(\Omega)\right\}\ \ \ \longrightarrow\ \ \ L^2_{\sqrt\rho}(\Omega)
\end{equation}
is then an isomorphism. Indeed, recall that the mapping
$$
\Delta:\ \ \ H^{2}(\Omega)\cap H^1_0(\Omega)\ \ \ \longrightarrow\ \ \ L^2(\Omega)
$$
is an isomorphism. Moreover, thanks to Theorem \label{a26-theo1}, the following mapping
$$
\Delta:\ \ \ \mathscr{Q}^{2}_{-1}(\Omega)\ \ \ \longrightarrow\ \ \ L^2_\rho(\Omega)
$$
is also an isomorphism. We conclude by using interpolation arguments and the following characterization
\begin{equation}\label{a28-e1bis}
\ds\left[H^2(\Omega)\cap H^1_0(\Omega), \mathscr{Q}^{2}_{-1}(\Omega)\right]_{1/2}\ =\ \left\{ v\in H^{3/2}(\Omega)\cap H^1_0(\Omega);\ \ \rho^{1/2}D^2 v\in\LG^2(\Omega)\right\}.
\end{equation}
\end{remark}


%
%
%
\medskip

\begin{lemma}\label{H3demi} For any function v  satisfying 
$$
v\in H^1(\Omega), \quad \sqrt \rho \Delta v \in L^2(\Omega) \quad and\quad v_{\vert\Gamma} \in H^1(\Gamma)
$$
we have the following properties
\begin{equation}\label{H3demiracine de rhograd2}
v\in H^{3/2}(\Omega) \quad and\quad\sqrt \rho\,  \nabla^2 v \in L^2(\Omega),
 \end{equation}
with the estimate
\begin{equation}\label{estimabe} 
 \Vert v\Vert_{H^{3/2}(\Omega)} + \Vert \sqrt \rho\,  \nabla^2 v\Vert_{L^2(\Omega)} \leq C(\Omega) (\Vert  v\Vert_{H^{1}(\Omega)} + \Vert \sqrt \rho  \Delta v\Vert_{L^2(\Omega)} + \Vert v \Vert_{H^{1}(\Gamma) }).
 \end{equation}
\end{lemma}

\begin{proof} From Theorem \ref{isodansEtheta}, we know that there exists a unique solution $u\in H^{3/2}_0(\Omega)$, with  $ \sqrt \rho\, \nabla^2 u \in L^2(\Omega)$ and satisfying $ \Delta u = \Delta v$ in $ \Omega $ and the following estimate
\begin{equation}\label{estimabe1} 
 \Vert u\Vert_{H^{3/2}(\Omega)} + \Vert \sqrt \rho\,  \nabla^2 u\Vert_{L^2(\Omega)} \leq C(\Omega) \Vert \sqrt \rho  \Delta v\Vert_{L^2(\Omega)}.
 \end{equation}
Setting now $w = v - u$, we get 
$$
w\in H^1(\Omega), \quad \Delta w = 0 \quad \mathrm{in}\; \Omega  \quad \mathrm{and}\quad w_{\vert\Gamma} \in H^1(\Gamma).
$$
We can then use Theorem \ref{PDH, lambda=0} Point i) to deduce the existence of a unique  $z \in H^{3/2}(\Omega)$ with  $ \sqrt \rho\, \nabla^2 z \in L^2(\Omega)$ solution of the problem
$$
\Delta z = 0  \quad \mathrm{in}\; \Omega \quad \mathrm{and}\quad z = w \quad  \mathrm{on}\; \Gamma.
$$
Moreover we have te estimate
\begin{equation}\label{estimabe2} 
 \Vert z\Vert_{H^{3/2}(\Omega)} + \Vert \sqrt \rho\,  \nabla^2 z\Vert_{L^2(\Omega)} \leq C(\Omega) \Vert w\Vert_{H^1(\Gamma)}.
 \end{equation}
As the harmonic function $z - w \in H^1_0(\Omega)$, it follows that $ z = w$ in $\Omega$ and also $w \in H^{3/2}(\Omega)$ with  $ \sqrt \rho\, \nabla^2 w \in L^2(\Omega)$. 

Finally, since $ v = w + u$ we deduce  the properties \eqref{H3demiracine de rhograd2} and thanks to \eqref{estimabe1} and \eqref{estimabe2} we get the required estimate \eqref{estimabe}.\end{proof}

The next result concerns the Neumann problem  $(\mathcal{P}_N)$, with $\lambda = 0$.

\begin{theorem}\label{PNH, lambda=0} $\mathrm{(}${\bf Existence in} $H^s(\Omega)$  {\bf  for} $(\mathcal{P}_N)$  {\bf with} $\lambda=0 \mathrm {)}$ Let $\frac{1}{2}\leq s\leq\frac{3}{2}$  and  $\lambda = 0$.

\noindent  i) Then for any 
$$ 
h\in H^{s- 3/2}(\Gamma)\quad with \quad < h, 1> = 0,
$$
Problem $(\mathcal{P}_N^H)$ has a unique solution $u\in H^{s}(\Omega)\cap L^2_0(\Omega)$ with the estimate 
\begin{equation*}\label{}
|| u||_{H^{s}(\Omega)} \leq\ C\, ||h||_{H^{s-3/2}(\Gamma)}.
\end{equation*}
Moreover,  if $s=1/2$, then $\sqrt \rho\,\nabla u \in  (L^2(\Omega))^{N},$ with the estimate
\begin{equation*}\label{}
\left|\left| u\right|\right|_{H^{1/2}(\Omega)} + \left\Vert \sqrt \rho\,\nabla u \right\Vert_{\textit{\textbf L}^2(\Omega)}\ \leq \ C\,\left|\left| h\right|\right|_{H^{-1}(\Gamma)}.
\end{equation*}
 If $s=3/2$, then $\sqrt \rho\,\nabla^2 u \in  (L^2(\Omega))^{N\times N},$ with the estimate 
\begin{equation*}\label{}
\left|\left| u\right|\right|_{H^{3/2}(\Omega)}+\left|\left| \sqrt \rho\,\nabla^2 u\right|\right|_{\textit{\textbf L}^2(\Omega)} \ \leq \ C\,\left|\left| h\right|\right|_{L^{2}(\Gamma)}.
\end{equation*}
ii) For 
$$
s=3/2, \quad f\in L^2(\Omega) \quad and\quad  h \in L^2(\Gamma)\quad with \quad \int_{\Omega} f \, + < h, 1> = 0,
$$ 
Problem $(\mathcal{P}_N)$ has a unique solution $u\in H^{3/2}(\Omega)\cap L^2_0(\Omega)$ with the estimate 
\begin{equation*}\label{d02-240118-e1cexbisBi}
\left|\left| u\right|\right|_{H^{3/2}(\Omega)}+\left|\left| \sqrt \rho\,\nabla^2 u\right|\right|_{\textit{\textbf L}^2(\Omega)} \ \leq \ C \left(\Vert f \Vert_{L^2(\Omega)} + \left|\left| h\right|\right|_{H^{1}(\Gamma)} \right).
\end{equation*}
\end{theorem}

\subsection{Case $\lambda \in \R^*$}

We will now study the case where $\lambda$ is a real number not equal to $ 0$. 

\begin{theorem}\label{th1} $\mathrm{(}${\bf Resolvent Boundary Estimates for} $(\mathcal{P}_N^H) \mathrm{)}$ Suppose $\lambda \in \R^*$. Then Problem $(\mathcal{P}_N^H)$ has a unique solution which satisfies 
\begin{equation*}
u\in H^{3/2}(\Omega), \; \sqrt \rho\,\nabla^2 u \in (L^2(\Omega))^{N\times N}\quad and \quad u\in H^1(\Gamma),
\end{equation*}
with the following estimate:
\begin{equation}\label{estim1}
\begin{array}{rl}
\vert \lambda \vert^{3/2} \,  \Vert u\Vert_{L^2(\Omega)} &  + \; \vert\lambda\vert \, \Vert \sqrt \rho \nabla u\Vert_{L^2(\Omega)} +\;   \vert \lambda \vert^{1/2}  \Vert  u\Vert_{H^1(\Omega)} + \Vert u\Vert_{H^{3/2}(\Omega)} \; + \\\\
& + \; \Vert \sqrt \rho  \nabla^2 u\Vert_{L^2(\Omega)}+ \vert \lambda \vert \, \Vert u\Vert_{L^2(\Gamma)}  + \Vert u  \Vert_{H^1(\Gamma)} \leq C \Vert  h\Vert_{L^2(\Gamma)} 
\end{array}
\end{equation}
which holds\\ 
{\bf i)} if $\int_{\Gamma} h = 0$.  In this case  $\int_{\Omega} u = 0$ and the constant $C$ depends only on $\Omega$ and not on $\lambda$,\\
or\\
{\bf ii)} if $\int_{\Gamma} h \not= 0$, for any $\vert\lambda \vert \ge \lambda_0$  with arbitrary real fixed number  $\lambda_0 > 0$. In this case  the constant $C$ depends on $\lambda_0$ and on $\Omega$.

\end{theorem}

\begin{proof} Clearly, Problem $(\mathcal{P}_N^H)$ has a unique solution $u\in H^1(\Omega)$. By Theorem \ref{PNH, lambda=0}, we deduce that  $u\in H^{3/2}(\Omega)$ and $\sqrt \rho\,\nabla^2 u \in (L^2(\Omega))^{N\times N}$. Moreover, multiplying by $u$ in $(\mathcal{P}_N^H)$, we get
\begin{equation}\label{egaliteenergie}
\lambda^2 \Vert u\Vert^2_{L^2(\Omega)} + \Vert \nabla u\Vert^2_{L^2(\Omega)} = \int_{\Gamma} uh .
\end{equation}

\noindent{\bf 1. Case} $\int_{\Gamma} h = 0$.

In this case we have $\int_{\Omega} u = 0$. Using Poincar\'e--Wirtinger inequality and the traces properties, we get from \eqref{egaliteenergie}
\begin{equation*}
\Vert u\Vert^2_{H^1(\Omega)} \leq C(\Omega) \Vert \nabla u\Vert^2_{L^2(\Omega)}  \leq C(\Omega) \Vert u\Vert_{L^2(\Gamma)} \Vert h\Vert_{L^2(\Gamma)} \leq C(\Omega) \Vert u\Vert_{H^1(\Omega))} \Vert h\Vert_{L^2(\Gamma)}.
\end{equation*}
From this last inequality and the relation \eqref{egaliteenergie}, we deduce that 
\begin{equation}\label{estimnablau}
\vert \lambda \vert \,  \Vert u\Vert_{L^2(\Omega)} + \Vert u\Vert_{H^1(\Omega)} \leq C(\Omega)  \Vert h\Vert_{L^2(\Gamma)},
\end{equation}
where the constant $C(\Omega)$ depends only on $\Omega$ and not on $\lambda$. 

We claim now that 
\begin{equation}\label{ab}
 \vert \lambda \vert \, \Vert u\Vert_{L^2(\Gamma)}  + \Vert\nabla_{\mathscr{T}}u  \Vert_{L^2(\Gamma)} \leq  C(\Omega) \Vert  h\Vert_{L^2(\Gamma)}.
\end{equation}
Which then implies four inequalities. The first one is given by
\begin{equation}\label{abb}
 \vert \lambda \vert \, \Vert u\Vert_{L^2(\Gamma)}  + \Vert u  \Vert_{H^1(\Gamma)} \leq  C(\Omega) \Vert  h\Vert_{L^2(\Gamma)}
\end{equation}
since $ \Vert u\Vert_{L^2(\Gamma)} \leq C(\Omega)  \Vert h\Vert_{L^2(\Gamma)}$ and also, after using  \eqref{egaliteenergie}, the following estimate 
\begin{equation}\label{estimabc}
\vert \lambda \vert^{3/2} \,  \Vert u\Vert_{L^2(\Omega)} + \vert \lambda \vert^{1/2}  \Vert  \nabla u\Vert_{L^2(\Omega)}   \leq  C(\Omega) \Vert  h\Vert_{L^2(\Gamma)}.
\end{equation}
For the third inequality, we multiply by $-\rho\Delta u$ the equation $\lambda^2 u - \Delta u = 0 $ and we integrate by parts:
\begin{equation}\label{energieavecd}
\lambda^2 \Vert \sqrt \rho \nabla u\Vert^2_{L^2(\Omega)} + \Vert \sqrt \rho \Delta u\Vert^2_{L^2(\Omega)} = -  \lambda^2 \int_{\Omega} u \nabla  \rho \cdot \nabla u\leq C(\Omega)  \lambda^2 \Vert u\Vert_{L^2(\Omega)} \Vert \nabla u\Vert_{L^2(\Omega)} 
\end{equation}
since $\nabla  \rho  \in L^\infty (\Omega)$. From \eqref{estimabc}, we then get the following inequality:
\begin{equation}\label{estimabd}
\vert\lambda\vert \Vert \sqrt \rho \nabla u\Vert_{L^2(\Omega)} + \Vert \sqrt \rho \Delta  u\Vert_{L^2(\Omega)} \leq C(\Omega) \Vert h\Vert_{L^2(\Gamma)}.
\end{equation}
Using then \eqref{estimnablau},    \eqref{abb} and  \eqref{estimabd} and Lemma \ref{H3demi}, we get the fourth inequality:
 \begin{equation}\label{estimabf} 
 \Vert u\Vert_{H^{3/2}(\Omega)} + \Vert \sqrt \rho\,   \nabla^2 u\Vert_{L^2(\Omega)} \leq C(\Omega)  \Vert h\Vert_{L^2(\Gamma)}.
 \end{equation}
 From the estimates  \eqref{abb}, \eqref{estimabc},  \eqref{estimabd}  and \eqref{estimabf} we finally get the required estimate \eqref{estim1}

To prove the estimate \eqref{ab}, we need five steps.

\noindent{\bf Step 1}. {\it We first suppose that $\Omega$ is $\mathcal{C}^{1,1}$ and $h\in H^{1/2}(\Gamma)$}.  \medskip

Then we know that the solution  $u\in H^2(\Omega)$ and we can use the following Rellich identity:
\begin{equation}\label{Rellich}
\begin{array}{rl}
\ds \int_\Gamma \textbf{\textit h}\cdot\textbf{\textit n}\left|\nabla_{\mathscr{T}}u\right|^2 &+\,  2 \ds \int_\Omega(\textbf{\textit h} \cdot\nabla u)\,\Delta u  = \ds \int_\Gamma \textbf{\textit h}\cdot\textbf{\textit n}\left|\frac{\partial u}{\partial \textbf{\textit n}}\right|^2
 + 2 \int_\Gamma \textbf{\textit h}_{\mathscr{T}}\cdot\nabla u\frac{\partial u}{\partial \textbf{\textit n}}\\\\
&\ds+\int_\Omega\left[ (\mathrm{div}\, \textbf{\textit h})\nabla u -  2 \frac{\partial u}{\partial x_k}\frac{\partial \textbf{\textit h}}{\partial x_k}\right]\cdot\nabla u,
\end{array}
\end{equation}
where $\textbf{\textit h}=\textbf{\textit h}_{\mathscr{T}}+(\textbf{\textit h}\cdot\textbf{\textit n})\textbf{\textit n}$  and $\textbf{\textit h}_{\mathscr{T}}$ is the tangential component of the vector $\textbf{\textit h}$. But,
\begin{equation*}
\int_{\Omega} (\textbf{\textit h} \cdot\nabla u)\,\Delta u  = \frac{1}{2} \lambda^2 \int_{\Omega} \textbf{\textit h} \cdot\nabla (u^2) = -  \frac{1}{2} \lambda^2 \int_{\Omega}  u^2  \mathrm{div}\, \textbf{\textit h} + \frac{1}{2} \lambda^2 \int_{\Gamma}  u^2   \textbf{\textit h}\cdot \textbf{\textit n},
\end{equation*}
so we get the following relation:
\begin{equation}\label{RellichBi}
\begin{array}{rl}
&\ds \int_\Gamma \textbf{\textit h}\cdot\textbf{\textit n} \left(\frac{1}{2}\lambda^2 u^2 +  \left|\nabla_{\mathscr{T}}u\right|^2\right) 
 = \ds \int_\Gamma \textbf{\textit h}\cdot\textbf{\textit n}\left|\frac{\partial u}{\partial \textbf{\textit n}}\right|^2
 + 2 \int_\Gamma \textbf{\textit h}_{\mathscr{T}}\cdot\nabla u\frac{\partial u}{\partial \textbf{\textit n}}\\\\
&\ds+\int_\Omega\left[ (\mathrm{div}\, \textbf{\textit h})\nabla u -  2 \frac{\partial u}{\partial x_k}\frac{\partial \textbf{\textit h}}{\partial x_k}\right]\cdot\nabla u + \frac{1}{2} \lambda^2 \int_{\Omega}  u^2  \mathrm{div}\, \textbf{\textit h}.
\end{array}
\end{equation}
Choosing  a vector field $\textbf{\textit h} \in \mathcal{C}^\infty(\overline{\Omega})$ such that $\textbf{\textit h}\cdot\textbf{\textit n} \ge \alpha > 0$ on $\Gamma$, the previous inequality implies that
\begin{equation*}\label{Rellichestim1}
\begin{array}{rl}
\ds \int_\Gamma (\lambda^2 u^2  +  \left|\nabla_{\mathscr{T}}u\right|^2) 
& \leq  C(\Omega) \Big(\Vert h \Vert^2_{L^2(\Gamma)} + \Vert \nabla_{\mathscr{T}}u\Vert_{L^2(\Gamma)}\Vert h \Vert_{L^2(\Gamma)} \, +\\\\
&  + \; \Vert \nabla u \Vert^2_{L^2(\Omega)} + \lambda^2 \Vert u \Vert^2_{L^2(\Omega)}\Big),
\end{array}
\end{equation*}
and finally thanks to \eqref{estimnablau} and \eqref{egaliteenergie}, we have
\begin{equation*}\label{Rellichestim2}
\ds \int_\Gamma (\lambda^2 u^2  +  \left|\nabla_{\mathscr{T}}u\right|^2) 
 \leq  C(\Omega) \Big(\Vert h \Vert^2_{L^2(\Gamma)} +  \Vert u \Vert^2_{L^2(\Gamma)}\Big) \leq C(\Omega) \Vert h \Vert^2_{L^2(\Gamma)},
\end{equation*}
where the constant $C(\Omega)$ depends only on the Lipschitz character of $\Omega$.\medskip

\noindent{\bf Step 2}. {\it We now do not suppose that  $\Omega$ is $\mathcal{C}^{1,1}$ but we suppose that $h\in H^{1/2}(\Gamma)$}.
 
To prove the estimate \eqref{ab}, it suffices to consider $\Omega$ of the form
 $$
\Omega\ = \ \left\{\,(\textit{\textbf x}',x_N)\in\R^N;\ \ x_N<\xi(\textit{\textbf x}') \right\},
$$
where $\xi\in \mathcal{C}^{0,1}(\R^{N-1})$ with supp $\xi$ compact. As in \cite{McLean}, we choose a sequence  $\xi_k\in \mathcal{C}^\infty(\R^{N-1})$ such that
\begin{enumerate}
\item  $\xi_k\  \longrightarrow \ \ \xi\ \textrm{ in }  L^\infty(\R^{N-1})\quad $ and $\quad \nabla\xi_k\  \longrightarrow \ \ \nabla\xi\ \textrm{ in }  L^p(\R^{N-1})$
 for any $1\leq p<\infty$
 \item $\xi_k\leq \xi  \textrm{ on } \R^{N-1}$ and $\left|\left|\, \nabla\xi_k\,\right|\right|_{L^\infty(\R^{N-1})}\leq C$ for any $k\geq 1$
 \item $\xi_k(\textit{\textbf x}')=\xi(\textit{\textbf x}')$ if $|\textit{\textbf x}'|\geq R_0$
\end{enumerate}
and we set 
$$
\Omega_k\ = \ \left\{\,(\textit{\textbf x}',x_N)\in\R^N;\ \ x_N<\xi_k(\textit{\textbf x}') \right\}.
$$

We define now the following sequence
$$
\textit{\textbf x}'\in \R^{N-1},\quad h_k(\textit{\textbf x}',\xi_k(\textit{\textbf x}'))= \ds\chi(x_N) h(\textit{\textbf x}',\xi(\textit{\textbf x}'))\frac{\theta(\textit{\textbf x}')}{\theta_k(\textit{\textbf x}')}
$$
where $\chi\in \mathcal{D}(\R)$ and $\chi=1$ in a neighbourhood  of $[\mathrm{min}\,\xi,\mathrm{max}\,\xi]$, $\theta$ and $\theta_k$ are defined as follow:
$$
\theta_k(\textit{\textbf x}')=(1+\vert\nabla\xi_k(\textit{\textbf x}')\vert^2)^{1/2},\ \quad\ 
\theta(\textit{\textbf x}')=(1+\vert\nabla\xi(\textit{\textbf x}')\vert^2)^{1/2}
$$
and observe that for any $1\leq p < \infty$ 
$$
1\leq \theta_k\leq C, \quad \theta_k\rightarrow \theta\ \textrm{ in } L^p(\R^{N-1}).
$$

It is clear that 
$$
h_k\in H^{1/2}(\Gamma_k)\quad \mathrm{and} \int_{\Gamma_{k}} h_k = 0.
$$
Let $u_k\in H^2(\Omega_k)\cap L^2_0(\Omega_k)$ be the unique solution of the following equation
\begin{equation*}\label{d03-080318-ce1sup}
\lambda^2 u_k - \Delta u_k  = 0\ \textrm{ in }\ \Omega_k,\quad \frac{\partial u_k}{\partial\textbf{\textit n}_k} = h_k \textrm{ on } \Gamma_k.
\end{equation*}
As above, we have
\begin{equation}\label{estimnablauk} 
\Vert u_k\Vert_{H^1(\Omega_k)}  \leq C(\Omega_k) \Vert h_k\Vert_{L^2(\Gamma_k)}\leq C(\Omega) \Vert h\Vert_{L^2(\Gamma)},
\end{equation}
since the constant in the Poincar\'e--Wirtinger inequality depends only on the Lipschitz character of $\Omega$ and
\begin{equation}\label{Rellichestim2k}
\ds \int_{\Gamma_k} (\lambda^2 u_k^2  +  \left|\nabla_{\mathscr{T}}u_k\right|^2) 
 \leq  C(\Omega_k) \, \Big(\Vert h_k \Vert^2_{L^2(\Gamma_k)} +  \Vert u_k \Vert^2_{L^2(\Gamma_k)}\Big) \leq C(\Omega) \Vert h \Vert^2_{L^2(\Gamma)}.
\end{equation}
After that, we extend $u_k$ outside of $\Omega$, denoted by $\widetilde{u_k}$, satisfying
$\widetilde{u_k}\in H^1(\Omega)$  and $\Vert\widetilde{u_k}\Vert_{H^1(\Omega\setminus\Omega_k)}\rightarrow 0$. Because of the estimate  \eqref{estimnablauk}, that means that  
\begin{equation*}\label{ukbornedansH1}
(\widetilde{u_k})_k\quad \mathrm{ is \,\, bounded\, \, in\,  }\quad H^1(\Omega).
\end{equation*}

\noindent {\bf Step 3}.  {\it We always assume that $h\in H^{1/2}(\Gamma)$ and we will prove that $\widetilde{u_k}\rightarrow u$ in $H^1(\Omega)$}. 

We start by observe that for any $\varphi\in H^1(\Omega)$, we have
\begin{equation}\label{d05-100318-e6}
\begin{array}{rl}
&\ds\int_{\Omega} \lambda^2 (\widetilde{u_k}-u) \varphi + \nabla(\widetilde{u_k}-u)\cdot\nabla\varphi\\\\ 
&\ds =  \left[\int_{\Omega}\lambda^2 \widetilde{u_k} \varphi + \nabla \widetilde{u_k}\cdot\nabla\varphi
-\int_{\Omega_k} \lambda^2u_k\varphi + \nabla u_k\cdot\nabla\varphi\right]  + \\\\     
&+ \ds  \left[ \int_{\Omega_k} \lambda^2u_k\varphi  + \nabla u_k\cdot\nabla\varphi
  -\int_{\Omega}\lambda^2u \varphi + \nabla u\cdot\nabla\varphi \right]\\\\
& \ds=  \int_{\Omega\setminus\Omega_k} \lambda^2\widetilde{u_k} \varphi + \nabla \widetilde{u_k}\cdot\nabla\varphi + \int_{\Gamma_k}h_k \varphi-\int_{\Gamma}h\varphi.
\end{array}
\end{equation}
Note that 
$$
\begin{array}{rlr}
&\ds\left\vert\int_{\Gamma_k}h_k\varphi- \int_{\Gamma}h\varphi\right\vert^2\\\\
&\ds \ = \ \left\vert\int_{\R^{N-1}}h(\textit{\textbf x}',\xi(\textit{\textbf x}'))\left[\varphi(\textit{\textbf x}',\xi_k(\textit{\textbf x}'))-
\varphi(\textit{\textbf x}',\xi(\textit{\textbf x}'))\right]\theta(\textit{\textbf x}')d\textit{\textbf x}'\right\vert^2\\\\
&\ds \leq C \left|\left|\,h\,  \right|\right|^2_{L^{2}(\Gamma)} \Vert \xi-\xi_k\Vert_{L^\infty(\R^{N-1})}\Vert \varphi\Vert^2_{H^1(\Omega\setminus\Omega_k)}.
\end{array}
$$
We have used above the following inequality:  for any $z\in H^1(\Omega)$, we have
$$
\begin{array}{rl}
\ds \left \vert z(\textit{\textbf x}',\xi(\textit{\textbf x}'))-z(\textit{\textbf x}',\xi_k(\textit{\textbf x}'))\right \vert\ = \, \ds  \left|\int_{\xi_k(\textit{\textbf x}')}^{\xi(\textit{\textbf x}')}\frac{\partial z}{\partial x_N}(\textit{\textbf x}',x_N) dx_N\right\vert \\\\
\ds\leq |\xi(\textit{\textbf x}')-\xi_k(\textit{\textbf x}')|^{1/2}\Big(\int_{\xi_k(\textit{\textbf x}')}^{\xi(\textit{\textbf x}')} \left|\frac{\partial z}{\partial x_N}(\textit{\textbf x}',x_N)\right|^2 dx_N\Big)^{1/2}
\end{array}
$$
and then by Cauchy-Schwarz inequality, we deduce
\begin{equation*}\label{d05-100318-e1}
\ds \int_{\R^{N-1}}|z(\textit{\textbf x}',\xi(\textit{\textbf x}'))-z(\textit{\textbf x}',\xi_k(\textit{\textbf x}'))|^2d\textit{\textbf x}'\leq
C\Vert \xi-\xi_k\Vert_{L^\infty(\R^{N-1})}\Vert \nabla z\Vert^2_{L^2(\Omega\setminus\Omega_k)}.
\end{equation*}

Choosing  $\varphi=\widetilde{u_k}- u\in H^1(\Omega)$ in \eqref{d05-100318-e6} we obtain, according to the boundedness of $\widetilde{u_k}$ in $H^1(\Omega)$, the following inequality 
\begin{equation*}\label{d05-100318-e7}
\begin{array}{rl}
&\ds \lambda^2\Vert \widetilde{u_k} - u\Vert^2_{L^2(\Omega)} +  \Vert \nabla (\widetilde{u_k} - u)\Vert^2_{L^2(\Omega)}
\leq\\\\
 &\ds C \Big[\left|\left|\,h\,\right|\right|_{L^2(\Gamma)}\Vert \xi-\xi_k\Vert^{1/2}_{{L^{\infty}(\R^{N-1})}} \Vert \nabla (\widetilde{u_k} - u)\Vert_{L^2(\Omega)}\; + \\\\
&+\, \lambda^2 \Vert\widetilde{u_k}\Vert_{L^2(\Omega\setminus\Omega_k)}\Vert \widetilde{u_k} - u\Vert_{L^2(\Omega)}
+\; \Vert \nabla \widetilde{u_k}\Vert_{L^2(\Omega\setminus\Omega_k)}  \Vert \nabla( \widetilde{u_k} - u)\Vert_{L^2(\Omega \setminus\Omega_k )}\Big].
\end{array}
\end{equation*}
In other words, 
\begin{equation*}\label{d05-100318-e7b}
\begin{array}{rl}
&\ds \vert \lambda \vert\, \Vert \widetilde{u_k} - u\Vert_{L^2(\Omega)} +  \Vert \nabla (\widetilde{u_k} - u)\Vert_{L^2(\Omega)}\ds
\leq\ C \Big[\left|\left|\,h\,\right|\right|_{L^2(\Gamma)}\Vert \xi-\xi_k\Vert^{1/2}_{{L^{\infty}(\R^{N-1})}} \; + \\\\
&+\; \vert \lambda\vert \, \Vert\widetilde{u_k}\Vert_{L^2(\Omega\setminus\Omega_k)}\Vert
+\, \Vert \nabla \widetilde {u_k}\Vert_{L^2(\Omega\setminus\Omega_k)}  \Big]
\end{array}
\end{equation*}
and then we get the strong convergence
\begin{equation*}\label{strongconv}
\widetilde{u_k}\rightarrow u \quad \mathrm{ in  }\quad H^1(\Omega).
\end{equation*}

\noindent{\bf Step 4}. {\it We always assume that $h\in H^{1/2}(\Gamma)$ and we will prove the estimate \eqref{ab}}. 

For that, we set
$$
\psi_k (\textit{\textbf x}') =  u_k(\textit{\textbf x}', \xi_k(\textit{\textbf x}')),\quad \psi (\textit{\textbf x}') =  u(\textit{\textbf x}', \xi(\textit{\textbf x}')).
$$
By using Lebesgue dominated convergence theorem, it is easy to prove that
$$
\psi_k \longrightarrow\ \psi  \textrm{ in } L^2(\R^{N-1}),
$$
which means in particular that
\begin{equation}\label{convuL2k}
\Vert u_k\Vert_{L^2(\Gamma_{k})} \rightarrow \Vert u\Vert_{L^2(\Gamma)}.
\end{equation}
Then 
$$
\Vert \psi_k \Vert_{H^1(\R^{N-1})}\leq C\Vert u_k \Vert_{H^1(\Gamma_{k})}\leq C \left(\Vert u_k \Vert_{L^2(\Gamma_{k})} + \Vert \nabla_{\tau} u_k \Vert_{L^2(\Gamma_{k})} \right).  
$$
By the estimate \eqref{Rellichestim2k}, we deduce
$$
\Vert \psi_k \Vert_{H^1(\R^{N-1})}\leq C \left(\Vert u \Vert_{L^2(\Gamma)} + \Vert h \Vert_{L^2(\Gamma)} \right) \leq C \Vert h \Vert_{L^2(\Gamma)}. 
$$
Hence $\psi \in H^1(\R^{N-1})$ and $ \psi_k \rightharpoonup \psi$ in $H^1(\R^{N-1})$ with the estimate
$$
\Vert u \Vert_{H^1(\Gamma)} = \Vert \psi \Vert_{H^1(\R^{N-1})}\leq \liminf_{k}\Vert \psi_k \Vert_{H^1(\R^{N-1})}\leq C\Vert h \Vert_{L^2(\Gamma)}. 
$$
The required estimate is then consequence of \eqref{convuL2k} and \eqref{Rellichestim2k}.

\noindent{\bf Step 5}. {\it We finally assume that $h\in L^{2}(\Gamma)$ (with the average on $\Gamma$ equal to 0) and we will prove the estimate \eqref{ab}}. 

For that, let $h_m \in H^{1/2}(\Gamma)\cap L^2_0(\Gamma)$ such that  $h_m \rightarrow h$ in $L^2(\Gamma)$. Let $u_m\in H^2(\Omega)\cap L^2_0(\Omega)$ the unique solution of Problem $ (\mathcal{P}^H_N)$ satisfying the estimate
\begin{equation*}\label{estim1bis}
  \Vert \lambda u_m\Vert_{L^2(\Gamma)} + \Vert\nabla_{\mathscr{T}}u_m  \Vert_{L^2(\Gamma)} \leq C \Vert  h_m\Vert_{L^2(\Gamma)}. 
\end{equation*}
We  get the estimate \eqref{estim1} by passing to the limit in the above inequality.

\noindent{\bf 2. Case} $\int_{\Gamma} h \not= 0$.

The proof is very similar. The only change concerns the estimate \eqref{estimnablau} which now holds when $\vert\lambda\vert \ge \lambda_0 > 0$, with the constant $C$ depending on $\Omega$ and on $\lambda_0$. \end{proof}

\begin{remark}\upshape If $\Omega$ is the half-space $\R^N_+$, then an explicit solution to  Problem $ (\mathcal{P}^H_N)$ is given by:

\begin{equation*}\label{v-fourier1}
u({\textit{\textbf x}}',x_N) =\mathcal{F}^{-1}_{\boldsymbol{\xi}'}\Big(- \frac{1}{\sqrt {\lambda^2 +  |\boldsymbol{\xi}'|^2}}e^{-x_N\sqrt{\lambda^2 + |\boldsymbol{\xi}'|^2}} \hat{h}(\boldsymbol{\xi}')\Big),\ \ \ x_N > 0.
\end{equation*}

\noindent \textit{i)} By using Plancherel's Theorem, we have
\begin{equation*}\label{b13-e1}
\begin{array}{rl}
||\, u \,||^2_{L^2(\R^N_+)}&=\ds\int_0^\infty\int_{\R^{N-1}}
\frac{1} {\lambda^2 +  |\boldsymbol{\xi}'|^2}e^{-2x_N\sqrt{\lambda^2+|\boldsymbol{\xi}'|^2}} |\hat{h}(\boldsymbol{\xi}')|^2\,d \boldsymbol{\xi}'d x_N\\\\
&=\ds\frac{1}{2}\int_{\R^{N-1}}\frac{|\hat{h}(\boldsymbol{\xi}')|^2}{(\lambda^2+|\boldsymbol{\xi}'|^2)^{3/2}} d \boldsymbol{\xi}',
\end{array}
\end{equation*}
\textit{i.e.}
\begin{equation*}\label{3-theta1-demi}
\vert \lambda \vert^{3/2}||\, u\,||_{L^2(\R^N_+)}\leq \frac{1}{\sqrt 2}\Vert \,h\,\Vert_{L^{2}(\R^{N-1})}.
\end{equation*}
\textit{ii)} Moreover, we have
\begin{equation*}\label{v-fourier3}
\begin{array}{rl}
||\, \nabla u\,||^2_{L^2(\R^N_+)}&=\ds\ds\int_0^\infty\int_{\R^{N-1}}\left(\vert \boldsymbol{\xi}'|^2|\mathcal{F}_{\boldsymbol{\xi}'}u(\boldsymbol{\xi}',x_N)|^2+\left| \ds\frac{\partial\mathcal{F}_{\textit{\textbf x}'}u}{\partial x_N} \right|^2\right)\,d \boldsymbol{\xi}' d x_N\\\\
&=\ds\ds\int_0^\infty\int_{\R^{N-1}}\left(\frac{|\boldsymbol{\xi}'|^2}{(\lambda^2+|\boldsymbol{\xi}'|^2} + 1\right)e^{-2x_N\sqrt{\lambda^2+|\boldsymbol{\xi}'|^2}} |\hat{h}(\boldsymbol{\xi}')|^2 \,d \boldsymbol{\xi}' d x_N\\\\
&=\frac{1}{2}\ds\int_{\R^{N-1}}\frac{\lambda^2 + 2|\boldsymbol{\xi}'|^2}{(\lambda^2+ |\boldsymbol{\xi}'|^2)^{3/2}}  \vert\hat{h}(\boldsymbol{\xi}')\vert^2 \,d \boldsymbol{\xi}' \\\\
&\leq \vert \lambda \vert^{-1}||\,h\,||^2_{L^{2}(\R^{N-1})}.
\end{array}
\end{equation*}
\textit{iii)} We verify easily that 
$$
\vert \lambda \vert\, \left\| u\right\|_{L^2(\R^{N-1})} + \sum_{j= 1}^{N-1} \left\| \frac{\partial u}{\partial x_j} \right\|_{L^2(\R^{N-1})}\leq 2\Vert \,h\,\Vert_{L^{2}(\R^{N-1})}.
$$ 

\noindent\textit{iv)} As
$$
||\, \nabla^2 u\,||_{L^2(\R^N_+)} \leq C \left( ||\,h\,||_{H^{1/2}(\R^{N-1})} + \vert \lambda  \vert \; ||\,h\,||_{L^{2}(\R^{N-1})}\right),
$$
by using an interpolation argument, we can prove that the semi-norm in $H^{3/2}(\R^N_+) $ of $u$ satisfies:
$$
\vert u \vert_{H^{3/2}(\R^N_+)}  \leq C_N ||\,h\,||_{L^{2}(\R^{N-1})}.
$$
\noindent\textit{v)} As above, we can show that
\begin{equation*}
\vert\lambda\vert \, \Vert \sqrt x_N \nabla u\Vert_{L^2(\R^N_+)} + \; \Vert \sqrt x_N  \nabla^2 u\Vert_{L^2(\\R^N_+)} \leq C_N \Vert  h\Vert_{L^2(R^{N-1})}. 
\end{equation*}

\end{remark}

\begin{remark}  \upshape
When $\Omega$ is $\mathcal{C}^{1,1}$, the proof can be simplified as follows. Let $(h_m)_m\subset H^{1/2}(\Gamma)$ be such that  $h_m \rightarrow h$ in $L^2(\Gamma)$ as $m\rightarrow \infty$. Let $u_m\in H^2(\Omega)$ be the unique solution of Problem $ (\mathcal{P}^H_N)$ with $ \lambda = \lambda_m$. Clearly the sequence $(u_m)_m$ is bounded in $H^1(\Omega)$ and then converges to $u$ weakly in $H^1(\Omega)$. Moreover 
$$
\frac{\partial u_m}{\partial\textit{\textbf n}} \rightharpoonup \frac{\partial u}{\partial\textit{\textbf n}} \quad \mathrm{in} \quad H^{-1/2}(\Gamma)
$$
where $u$ satisfies $ (\mathcal{P}^H_N)$. As $u_m$ satisfies the estimate 
\begin{equation*}\label{estim1bisBi}
  \Vert \lambda u_m\Vert_{L^2(\Gamma)} + \Vert\nabla_{\mathscr{T}}u_m  \Vert_{L^2(\Gamma)} \leq C \Vert  h_m\Vert_{L^2(\Gamma)} 
\end{equation*}
we get our result by passing to the limit in the above estimate.
\end{remark}

\begin{remark}\upshape  If $\int_{\Gamma} h \not= 0$, there is not possible to obtain the estimate \eqref{estim1} for small values of $\vert\lambda\vert$. Suppose the contrary, then we have in particular
$$
 \vert \lambda \vert\, (\Vert \nabla u\Vert_{L^2(\Omega)} +  \Vert u\Vert_{L^2(\Gamma)} )\leq C \Vert h \Vert_{L^2(\Gamma)}.
$$
We know that the solution of  Problem $ (\mathcal{P}^H_N)$ satisfies the relation 
$$
\int_{\Omega} u = \frac{1}{\lambda^2} \int_{\Gamma} h
$$
and so we have 
$$
 \Vert u\Vert_{L^2(\Omega)} \geq \frac{C(\Omega)}{ \lambda^2}\left| \int_{\Gamma} h\right|.
 $$
As 
 $$
 \Vert u\Vert_{L^2(\Omega)} \leq C (\Omega) \left(\Vert \nabla u\Vert_{L^2(\Omega)} + \Vert u\Vert_{L^2(\Gamma)} \right),
 $$
we can then deduce the following inequality: for any small values of $\vert\lambda\vert$
$$
\left|\int_{\Gamma} h\right|  \leq C(\Omega) \vert\lambda \vert \, \Vert h \Vert_{L^2(\Gamma)},
$$
which means that $\int_{\Gamma} h = 0$.

\end{remark}

We consider now the same question for the Dirichlet problem $(\mathcal{P}_D)$.

\begin{theorem}\label{th2} $\mathrm{(}${\bf Resolvent Boundary Estimates for} $(\mathcal{P}_D^H) \mathrm{)}$ Let $\Omega$ be a Lipschitz bounded open subset  of $\R^N$, $g\in H^1(\Gamma)$ and suppose $\lambda \in \R^*$. Then Problem $(\mathcal{P}_D^H)$ has a unique solution which satisfies 
\begin{equation*}\label{prosolPDH}
u\in H^{3/2}(\Omega),\quad \sqrt \rho \,  \nabla^2 u \in L^2(\Omega), \quad \frac{\partial u}{\partial\textit{\textbf n}}\in L^2(\Gamma),
\end{equation*}
with the following estimate:
\begin{equation}\label{estim3}
\begin{array}{rl}
\vert \lambda \vert^{3/2} \,  \Vert u\Vert_{L^2(\Omega)} & + \; \vert\lambda\vert \, \Vert \sqrt \rho \nabla u\Vert_{L^2(\Omega)}  +\;   \vert \lambda \vert^{1/2}  \Vert  u\Vert_{H^1(\Omega)} + \Vert u\Vert_{H^{3/2}(\Omega)} + \\\\
&   + \; \Vert \sqrt \rho  \nabla^2 u\Vert_{L^2(\Omega)} + \left\| \frac{\partial u}{\partial \textbf{\textit n}}\right\|_{L^2(\Gamma)} \leq C(\Omega) \left( \vert \lambda \vert \; \Vert  g \Vert_{L^2(\Gamma)}  + \Vert g \Vert_{H^1(\Gamma)} \right).
\end{array}
\end{equation}
\end{theorem}

\begin{proof} Clearly, Problem $(\mathcal{P}_D^H)$ has a unique solution $u\in H^1(\Omega)$. By Theorem \ref{Necas Property}, we know that $\frac{\partial u}{\partial \textbf{\textit n}}\in L^2(\Gamma)$ and by Theorem \ref{PDH, lambda=0} that  $u\in H^{3/2}(\Omega)$ and $  \nabla^2 u \in L^2(\Omega)$. Moreover, multiplying by $u$, we get the following relation
\begin{equation*}\label{egaliteenergiebis}
\lambda^2 \Vert u\Vert^2_{L^2(\Omega)} + \Vert \nabla u\Vert^2_{L^2(\Omega)} = \int_{\Gamma} g\frac{\partial u}{\partial \textbf{\textit n}}.
\end{equation*}

To prove the estimate \eqref{estim3}, we need six steps.

\noindent{\bf Step 1}. {\it We first suppose that $\Omega$ is $\mathcal{C}^{1,1}$ and $g\in H^{3/2}(\Gamma)$}.  \medskip

Then we know that the solution  $u\in H^2(\Omega)$ and thanks to the Rellich identity \eqref{RellichBi} and \eqref{egaliteenergiebis}, we have the following estimate:

\begin{equation*}
\ds \int_\Gamma \left|\frac{\partial u}{\partial \textbf{\textit n}}\right|^2 \leq C(\Omega)\left( \ds \int_\Gamma (\lambda^2 g^2  + \left|\nabla_{\mathscr{T}}g\right|^2) + (\Vert g \Vert_{L^2(\Gamma)}  +\Vert \nabla_{\mathscr{T}} g \Vert_{L^2(\Gamma)}) \left\| \frac{\partial u}{\partial \textbf{\textit n}}\right\|_{L^2(\Gamma)} \right),
\end{equation*}\\
from which we deduce that
\begin{equation}\label{RellichD}
\left\| \frac{\partial u}{\partial \textbf{\textit n}}\right\|_{L^2(\Gamma)} \leq C(\Omega)\left(\vert \lambda \vert\Vert g \Vert_{L^2(\Gamma)} + \Vert g\Vert_{H^1(\Gamma)} \right),
\end{equation}
where the constant $C(\Omega)$ depends only on the Lipschitz character of $\Omega$. 

 \noindent{\bf Step 2}. {\it We now do not suppose that  $\Omega$ is $\mathcal{C}^{1,1}$ but we assume that $g\in H^{3/2}(\Gamma)$}. 
 \medskip 
 
 To prove the estimate \eqref{RellichD} for $\Omega$ only  Lipschitz, it suffices to consider again, as in the proof of Theorem \ref{th1}, $\Omega$ of the form
$$
\Omega\ = \ \left\{\,(\textit{\textbf x}',x_N)\in\R^N;\ \ x_N<\xi(\textit{\textbf x}') \right\},
$$
where $\xi\in \mathcal{C}^{1,1}(\R^{N-1})$ with supp $\xi$ compact.   

Let us consider now the following function:
\begin{equation}\label{condlim}
\quad\textit{\textbf x}=(\textit{\textbf x}',x_N)\in \Omega;\quad g_k(\textit{\textbf x})= \chi(x_N) u(\textit{\textbf x}',\xi_k(\textit{\textbf x}'))
\end{equation}

where $\chi\in \mathcal{D}(\R)$ and $\chi=1$ in a neighbourhood  of $[\mathrm{min}\,\xi,\mathrm{max}\,\xi]$. As $\Delta u\in L^2(\Omega)$, then $u\in H^2_{loc}(\Omega)$ and $g_k$ belongs to $ H^2(\Omega)$ with the estimate
$$ 
\left|\left|\,  g_k\,\right|\right|_{H^1(\Omega)} + \left|\left|\,  g_k\,\right|\right|_{H^1(\Gamma_k)}\ \leq\ C\left|\left|\,  u\,\right|\right|_{H^1(\Gamma)}\, \leq\ C\left|\left|\,  g\,\right|\right|_{H^1(\Gamma)},
$$
where $C$ does not depend on $k$. Because the regularity of $\Omega_k$, there exists a unique solution $u_k\in H^2(\Omega_k)$ satisfying
\begin{equation*}\label{d03-080318-ce1}
\lambda^2 u_k -\Delta u_k  = 0\ \textrm{ in }\ \Omega_k\quad \mathrm{and}\quad  u_k = g_k \textrm{ on } \Gamma_k.
\end{equation*}
Multiplying by $u_k$ and using  the estimate \eqref{RellichD}, we get
\begin{equation*}
 \begin{array}{rl}
\vert \lambda \vert^2\left|\left|\,   u_k\,\right|\right|^2_{L^2(\Omega_k)} + \left|\left|\,   \nabla u_k\,\right|\right|^2_{L^2(\Omega_k)}&\leq \ds \left|\left|\,\frac{\partial u_k}{\partial \textit{\textbf n}_k}\,\right|\right|_{H^1(\Gamma_k)} \Vert g_k \Vert_{L^2(\Gamma_k)}\\\\
& \leq C(\Omega_k)(\vert \lambda \vert \, \Vert g_k \Vert_{L^2(\Gamma_k)} + \Vert g_k \Vert_{H^1(\Gamma_k)})\Vert g_k \Vert_{L^2(\Gamma_k)}\\\\
& \leq C(\Omega) \left(\vert \lambda \vert \, \Vert g \Vert^2_{L^2(\Gamma)} + \Vert g \Vert^2_{H^1(\Gamma)} \right)
 \end{array}
 \end{equation*}
where $C(\Omega)$ does not depend on $k$ and depends only on the Lipschitz character of $\Omega$. Set now 
$$
\widetilde{u_k} = \begin{cases}  u_k\quad \mathrm{in}\; \Omega_k,\\
g_k \quad  \mathrm{in}\; \Omega\setminus \Omega_k.
\end{cases}
$$
It is clear that $\widetilde{u_k} \in H^1(\Omega)$ and
\begin{equation}\label{bounded}
(\widetilde{u_k})_k \quad \mathrm{is\, \, bounded \, \,  in }\quad  H^1(\Omega).
\end{equation}
 
 \noindent{\bf Step 3}.  {\it We always assume that $g\in H^{3/2}(\Gamma)$ and we will prove that $\widetilde{u_k}\rightarrow u$ in $H^1(\Omega)$}.   

For any $\varphi\in H^1(\Omega)$, we have
\begin{equation}\label{d04-090318-e5}
\ds\int_{\Omega_k} u_k \varphi + \nabla u_k\cdot\nabla\varphi\ =\ \ds\int_{\Gamma_k}\frac{\partial u_k}{\partial\textit{\textbf n}_k}\varphi 
\end{equation}
and
\begin{equation}\label{d04-090318-e6}
\ds\int_{\Omega}u \varphi + \nabla u\cdot\nabla\varphi\ =\ \ds \int_{\Gamma}\frac{\partial u}{\partial\textit{\textbf n}}\varphi.
\end{equation}
Consequently, we have for any $\varphi\in H^1(\Omega)$
\begin{equation*}\label{FormGreenpart}
\begin{array}{rl}
&\ds\int_{\Omega} (\widetilde{u_k}-u) \varphi + \nabla(\widetilde{u_k}-u)\cdot\nabla\varphi\\\\ 
&\ds =  \left[\int_{\Omega} (\widetilde{u_k} \varphi + \nabla \widetilde{u_k}\cdot\nabla\varphi)
-\int_{\Omega_k} (u_k\varphi + \nabla u_k\cdot\nabla\varphi)\right]  + \\\\     
&+ \ds  \left[ \int_{\Omega_k} (u_k\varphi  + \nabla u_k\cdot\nabla\varphi)
  -\int_{\Omega}(u \varphi + \nabla u\cdot\nabla\varphi) \right]\\\\
&\ds=  \int_{\Omega\setminus\Omega_k} (\widetilde{u_k} \varphi + \nabla \widetilde{u_k}\cdot\nabla\varphi)  + \,\int_{\Gamma_k}\frac{\partial u_k}{\partial\textit{\textbf n}_k}\varphi  - \int_{\Gamma}\frac{\partial u}{\partial\textit{\textbf n}}\varphi\,   .
\end{array}
\end{equation*}
So taking $\varphi =  \widetilde{u_k} - u$ and note that its restriction to $\Omega_k$ belongs to $ H^1_0(\Omega_k)$, so we deduce that
\begin{equation}\label{estimuktile-u}
\Vert  \widetilde{u_k} - u\Vert^2_{H^1(\Omega)} \leq \,  \ds C\Vert  \widetilde{u_k} - u\Vert_{H^1(\Omega)} \Vert  \widetilde{u_k} \Vert_{H^1(\Omega\setminus\Omega_k)}   +  \left|\left|\,\frac{\partial u}{\partial \textit{\textbf n}}\,\right|\right|_{L^2(\Gamma)} \Vert  \widetilde{u_k} - u\Vert_{L^2(\Gamma)} .
\end{equation}
But
\begin{equation*}\label{InegL2Gammak}
\begin{array}{rl}
\!\!\!\!\Vert \widetilde{u_k} - u\Vert^2_{L^2(\Gamma)}\!\!\!\! & = \ds \int_{\R^{N-1}}\vert u(\textit{\textbf x}',\xi_k(\textit{\textbf x}')) - u(\textit{\textbf x}',\xi(\textit{\textbf x}')\vert^2 \theta(\textit{\textbf x}') d\textit{\textbf x}'\\\\
\!\!\!\! & \leq C \ds\int_{\R^{N-1}} (\xi(\textit{\textbf x}') - \xi_k(\textit{\textbf x}') )    \int_{\xi_k(\textit{\textbf x}')}^{\xi(\textit{\textbf x}')} \left|\frac{\partial u}{\partial x_N}(\textit{\textbf x}',x_N)\right|^2 dx_N  d\textit{\textbf x}'\\\\
\!\!\!\! &\leq C  \Vert \xi-\xi_k\Vert_{L^\infty(\R^{N-1})}\Vert u \Vert^2_{H^1(\Omega\setminus\Omega_k)}. 
\end{array}
\end{equation*}
Using \eqref{bounded}, \eqref{estimuktile-u} and the convergence to 0 of the norm $ \Vert  \widetilde{u_k} \Vert_{H^1(\Omega\setminus\Omega_k)}$, we deduce the strong convergence in $H^1(\Omega)$ of $\widetilde{u_k} $ to $u$ as claimed.

\medskip

\noindent {\bf Step 4}. {\it We always assume that $g\in H^{3/2}(\Gamma)$ and we will prove the estimate \eqref{RellichD}}.

We start by remark that the following sequence
 $$
 \psi_k(\textit{\textbf x}')\ = \ \nabla u_k\cdot\textit{\textbf n}_k(\textit{\textbf x}',\xi_k(\textit{\textbf x}'))
 $$
is bounded in $L^2(\R^{N-1})$ since
$$
\ds\Vert \psi_k\Vert_{L^2(\R^{N-1})}\ \leq C \left\|\frac{\partial u_k}{\partial n_k}\right\|_{L^2(\Gamma_k)},
$$
where $C$ does not depend on $k$ because $(\xi_k)_k$ is bounded in $W^{1,\infty}(\R^{N-1})$  and from Step 1 we know that
\begin{equation}\label{estimI}
\left\|\frac{\partial u_k}{\partial n_k}\right\|_{L^2(\Gamma_k)} \leq C(\Omega_k)\Big(\vert \lambda\,  \vert\Vert g_k \Vert_{L^2(\Gamma)} + \Vert g_k\Vert_{H^1(\Gamma)} \Big) \leq C(\Omega)\Big(\vert \lambda \,  \vert\Vert g \Vert_{L^2(\Gamma)} + \Vert g\Vert_{H^1(\Gamma)} \Big).
\end{equation}
So, after passing to a subsequence, we can assume that
$$
\psi_k\ \rightharpoonup\ \psi \ \ \textrm{ in } L^2(\R^{N-1}).
$$
Define now 
$$
\widetilde\psi(\textit{\textbf x}',\xi(\textit{\textbf x}'))=\psi(\textit{\textbf x}')
$$
and let $\varphi\in H^1(\Omega)$. Then as $\theta_k\ \rightarrow\ \theta$ in $L^2(\R^{N-1})$, we have
\begin{equation*}\label{d04-090318-e4}
\begin{array}{rl}
\ds\int_{\Gamma_k}\frac{\partial u_k}{\partial\textit{\textbf n}_k}\varphi\ &\ds =\ \int_{\R^{N-1}}\psi_k(\textit{\textbf x}')\varphi(\textit{\textbf x}',\xi_k(\textit{\textbf x}'))\theta_k(\textit{\textbf x}')\, d\textit{\textbf x}'\\\\
&\ds \longrightarrow \int_{\R^{N-1}}\psi(\textit{\textbf x}')\varphi(\textit{\textbf x}',\xi(\textit{\textbf x}'))\theta(\textit{\textbf x}')\, d\textit{\textbf x}'\ =\ \int_\Gamma\widetilde\psi\varphi\quad \mathrm{as}\; k\rightarrow \infty.
\end{array}
\end{equation*}
Here, note that 
\begin{equation*}
\begin{array}{rl}
\Vert \widetilde\psi \Vert_{L^2(\Gamma)} \leq C \Vert \psi \Vert_{L^2(\R^{N-1})} &\leq C(\Omega) \liminf_{k\rightarrow \infty} \Vert \psi_k \Vert_{L^2(\R^{N-1})}\\
& \leq C(\Omega)\Big(\vert \lambda \vert \, \Vert g \Vert_{L^2(\Gamma)} + \Vert g\Vert_{H^1(\Gamma)}  \Big),
\end{array}
\end{equation*}
where we used the estimate \eqref{estimI}.

\medskip
 
Sending $k \rightarrow \infty$ in \eqref{d04-090318-e5} gives
\begin{equation*}\label{d04-090318-e10}
\ds\int_{\Omega}u\varphi + \nabla u\cdot\nabla\varphi\ =\ \ds \int_\Gamma\widetilde\psi\varphi
\end{equation*}
and thanks to \eqref{d04-090318-e6}, we get $\ds\frac{\partial u}{\partial\textit{\textbf n}}=\widetilde\psi$ belonging $L^2(\Gamma)$ and the required estimate \eqref{RellichD}.

\noindent{\bf Step 5}. {\it We assume that $g\in H^{1}(\Gamma)$ and we will prove the estimate \eqref{RellichD}}. 

For that, let $g_m \in H^{3/2}(\Gamma)$ such that  $g_m \rightarrow g$ in $H^1(\Gamma)$. Let $u_m\in H^{3/2}(\Omega)$ the unique solution of Problem $ (\mathcal{P}^H_D)$, with the Dirichlet boundary condition $u_m = g_m$ on $\Gamma$. From above we deduce the following estimate
\begin{equation*}\label{RellichDm}
\left\| \frac{\partial u_m}{\partial \textbf{\textit n}}\right\|_{L^2(\Gamma)} \leq C(\Omega)\Big(\vert \lambda \vert\Vert g_m \Vert_{L^2(\Gamma)} + \Vert g_m\Vert_{H^1(\Gamma)} \Big) \leq C(\Omega)\Big(\vert \lambda \vert\, \Vert g \Vert_{L^2(\Gamma)} + \Vert g\Vert_{H^1(\Gamma)} \Big).
\end{equation*}
We  get the estimate \eqref{RellichD} by passing to the limit.

 \noindent{\bf Step 6}. {\it We finally prove the estimate \eqref{estim3}}

This estimate is an immediate consequence of \eqref{estim1} and \eqref{RellichD}. \end{proof}

\begin{remark}\upshape In the case where $\Omega$ is the half-space $\R^N_+$, then an explicit solution to  Problem $ (\mathcal{P}^H_N)$ is given by:

\begin{equation*}\label{v-fourier1bis}
u({\textit{\textbf x}}',x_N) =\mathcal{F}^{-1}_{\boldsymbol{\xi}'}\Big(e^{-x_N\sqrt{\lambda^2 + |\boldsymbol{\xi}'|^2}} \hat{g}(\boldsymbol{\xi}')\Big),\ \ \ x_N > 0.
\end{equation*}
{\bf i)} Simple calculations give
\begin{equation*}
||\, u \,||^2_{L^2(\R^N_+)} =\ds\int_0^\infty\int_{\R^{N-1}}
\frac{1} {2\sqrt{\lambda^2 +  |\boldsymbol{\xi}'|^2}} \vert\hat{g}(\boldsymbol{\xi}')\vert^2\,d \boldsymbol{\xi}'d x_N
\end{equation*}
and so the estimate
\begin{equation*}\label{estimavecgL2Gam}
\vert \lambda \vert^{1/2}||\, u\,||_{L^2(\R^N_+)}\leq \frac{1}{\sqrt 2}\Vert \,g\,\Vert_{L^{2}(\R^{N-1})}.
\end{equation*} 
Likewise
\begin{equation*}
||\, \nabla u\,||^2_{L^2(\R^N_+)}= \frac{1}{2}\ds\int_{\R^{N-1}}\big(\vert \boldsymbol{\xi}'\vert^2  + \lambda^2\big)^{1/2}\vert \hat{g}(\boldsymbol{\xi}')\vert^2 
+  \frac{1}{2}\ \int_{\R^{N-1}}\frac{\vert \boldsymbol{\xi}' \vert^2} {\vert \boldsymbol{\xi}'\vert^2  + \lambda^2)^{1/2} }\vert \hat{g}(\boldsymbol{\xi}')\vert^2 
\end{equation*} 
and then
\begin{equation*}
\begin{array}{rl}
||\, \nabla u\,||^2_{L^2(\R^N_+)} &\leq \frac{1}{2}\left(\vert g\vert^2_{H^{1/2}(\R^{N-1})} + \vert \lambda\vert \, \Vert g\Vert^2_{L^{2}(\R^{N-1})}  +\vert \lambda\vert^{-1}  \vert g\vert^2_{H^{1}(\R^{N-1})}\right)\\\\
&\leq C \left( \vert \lambda\vert \, \Vert g\Vert^2_{L^{2}(\R^{N-1})}  +\vert \lambda\vert^{-1}  \Vert g\Vert^2_{H^{1}(\R^{N-1})}\right),
\end{array}
\end{equation*} 
where the notation $\vert \cdot \vert_{H^{s}(\R^{N-1})}$ denotes the semi-norm in $H^{s}(\R^{N-1})$.\\
The third estimate is the following:
\begin{equation*}
\left\|\frac{\partial u}{\partial x_N}\right\|^2_{L^2(\R^{N-1})} = \ds\int_{\R^{N-1}}(\vert \lambda\vert^2+|\boldsymbol{\xi}'|^2) |\hat{g}(\boldsymbol{\xi}')|^2d \boldsymbol{\xi}',
\end{equation*}
\textit{i.e},
\begin{equation*}\label{3-theta1-demibis}
\left\|\frac{\partial u}{\partial x_N}\right\|^2_{L^2(\R^{N-1})} = \vert \lambda\vert^2 \Vert\, g  \,\Vert^2_{L^2(\R^{N-1})} + \vert g \vert^2_{H^1(\R^{N-1})}.
\end{equation*}
With these relations, we can see that the estimates obtained in Theorem \ref{th2} are optimal.

\noindent{\bf ii)} In \eqref{estim3}, the estimate  is with respect the norm of $g$ in $H^1(\Gamma)$. We will see later $\mathrm{(}$see Theorem \ref{th2Bii}$\mathrm{)}$ a very weak estimate  with respect the norm of $g$ in $L^2(\Gamma)$ like the estimate \eqref{estimavecgL2Gam} above.

\end{remark}

\begin{remark} \label{rellich_id}\upshape
For the validity of the Rellich identity in Lipschitz domains with minimal regularity of the solution and which is contained in Ne\v{c}as's book \cite{Necas} but we note that the given proof is very incomplete and difficult to read. On the other hand we corrected some misprints in the proof given by 
McLean in \cite{McLean} (Theorem 4.24). Unlike our choice for the boundary condition for approximate solutions $\mathrm{(}$see \eqref{condlim}$\mathrm{)}$, Mclean has considered the following boundary condition 
\begin{equation}\label{condlimbis}
\quad\textit{\textbf x}=(\textit{\textbf x}',x_N)\in \Omega;\quad g_k(\textit{\textbf x})= u(\textit{\textbf x}',\xi(\textit{\textbf x}'))
\end{equation}
which unfortunately does not provide the $H^2$-regularity needed for his proof $\mathrm{(}$see page 151, line 3 and also the erratum (number 151) given in the followink link\\ https://web.maths.unsw.edu.au/~mclean/errata.pdf
\end{remark}

\section{The resolvent estimates. Complex Case} \label{sec3}
\setcounter{equation}{0}

We now are interested in the case $\lambda \in \mathbb{C},$ Re $\lambda > 0$.  

\subsection{Neumann Case}

Clearly, as for Theorem \ref{th1}, the following existence result holds.

\begin{theorem}\label{th3} $\mathrm{(}${\bf Existence, Uniqueness and Regularity for $(\mathcal{P}_N^H)\mathrm{)}$} Let  $h\in L^2(\Gamma)$ and suppose $\lambda \in\mathbb{C} $ with  $\mathrm{Re}\, \lambda  > 0$. Then Problem $(\mathcal{P}_N^H)$ has a unique solution which satisfies 
\begin{equation}\label{regularite}
u\in H^{3/2}(\Omega),\quad \sqrt {d} \, \nabla^2 u \in {\textbf{\textit L}}^2(\Omega) \quad and \quad u\in H^1(\Gamma).
\end{equation}
\end{theorem}

\begin{theorem}\label{th31} $\mathrm{(}${\bf Resolvent Boundary Estimates for $(\mathcal{P}_N^H)$} Let $h\in L^2(\Gamma)$ and suppose $\lambda \in\mathbb{C} $ with  $\mathrm{Re}\, \lambda \ge \omega $ for some arbitrary $\omega > 0$. Then for any real number $ r < 1/2$, there exists a positive constant $C_r$, depending only on $r$, $\Omega$ and $\omega$, such that the solution given by the previous theorem satisfies the following estimate
\begin{equation}\label{estim1bbisr}
\vert \lambda \vert^{3r} \,  \Vert u\Vert_{L^2(\Omega)}  +\;   \vert \lambda \vert^{r}  \Vert  u\Vert_{H^1(\Omega)}  + \;  \vert \lambda \vert^{2r} \, \Vert u\Vert_{L^2(\Gamma)}  + \Vert \nabla_{\mathscr{T}}u \Vert_{L^2(\Gamma)} \leq C_r \Vert  h\Vert_{L^2(\Gamma)}.
\end{equation}
\end{theorem}

\begin{proof} {\bf Step 1.} Setting $\lambda = \vert \lambda \vert e^{i\varphi}$, with $\vert \varphi\vert < \pi/2$, by vertue oh the previous theorem it suffices to suppose that $\vert \varphi\vert$ is close to $\pi/2$ and then $\vert sin\, \varphi \vert \ge \frac{1}{2} $. We know that the solution $u$ satisfies
\begin{equation}\label{egaliteenergiebbis}
\lambda^2 \Vert u\Vert^2_{L^2(\Omega)} + \Vert \nabla u\Vert^2_{L^2(\Omega)} = \int_{\Gamma} \overline{u}h 
\end{equation}
and then
\begin{equation*}\label{egaliteenergiebbisReal}
 \Vert \nabla u\Vert^2_{L^2(\Omega)} + Re( \lambda^2) \Vert u\Vert^2_{L^2(\Omega)} = Re \left( \int_{\Gamma} \overline{u}h \right)
\end{equation*}
and
\begin{equation*}\label{egaliteenergiebbisIm}
Im(\lambda^2) \Vert u\Vert^2_{L^2(\Omega)}  = Im \left(\int_{\Gamma} \overline{u}h \right). 
\end{equation*}
As 
$$
\vert Im(\lambda^2)\vert = \vert \lambda\vert^2 \vert sin(2\varphi)\vert \ge  \vert \lambda\vert \, \omega.
$$
Using now the relation \eqref{egaliteenergiebbis}, we obtain for any $\alpha > 0$
\begin{equation*}\label{egaliteenergiebbissum}
 \Vert \nabla u\Vert^2_{L^2(\Omega)} + \left( Re( \lambda^2)  + \alpha \vert Im( \lambda^2\vert \right) \Vert u\Vert^2_{L^2(\Omega)} \leq (1 + \alpha) \Vert u\Vert _{L^2(\Gamma)}\Vert h\Vert _{L^2(\Gamma)}.
\end{equation*}
Observe that there exist $\alpha$ and $c > 0$, depending only on $\omega$ such that 
$$
 Re (\lambda^2)  + \alpha \vert Im( \lambda^2\vert) \ge c \vert \lambda \vert.
$$
So we deduce that there exists a consatnt $C > 0$, which depends only on $\omega$  such that
\begin{equation}\label{estimgrad2+lamdau}
 \Vert \nabla u\Vert^2_{L^2(\Omega)} + \vert \lambda \vert \, \Vert u\Vert^2_{L^2(\Omega)} \leq C \Vert u \Vert_{L^2(\Gamma)}\Vert h \Vert_{L^2(\Gamma)}.
\end{equation}
But, we know that (see the proof of Theorem 1.5.1.10 in Grisvard \cite{grisvard})
\begin{equation}\label{inegGrisvard}
\Vert u \Vert_{L^2(\Gamma)} \leq C \left(\Vert u\Vert^{1/2}_{L^2(\Omega)}\Vert \nabla u\Vert^{1/2}_{L^2(\Omega)} + \Vert u \Vert_{L^2(\Omega)} \right).
\end{equation}
From these last relations, we get
\begin{equation*}
\begin{array}{rl}
 \Vert \nabla u\Vert^2_{L^2(\Omega)} + & \vert \lambda \vert \, \Vert u\Vert^2_{L^2(\Omega)} \leq C \left(\Vert u\Vert^{1/2}_{L^2(\Omega)}\Vert \nabla u\Vert^{1/2}_{L^2(\Omega)} + \Vert u \Vert_{L^2(\Omega)} \right)\Vert h \Vert_{L^2(\Gamma)}\\\\
 & = C \left(\vert \lambda \vert^{1/4}\Vert u\Vert^{1/2}_{L^2(\Omega)}\vert \lambda \vert^{-1/4}\Vert \nabla u\Vert^{1/2}_{L^2(\Omega)}\Vert h \Vert_{L^2(\Gamma)} + \vert \lambda \vert^{1/2} \Vert u \Vert_{L^2(\Omega)} \vert \lambda \vert^{-1/2}\Vert h \Vert_{L^2(\Gamma)}\right).
\end{array}
\end{equation*}
By Young inequality, it follows that
\begin{equation*}
\begin{array}{rl}
 \Vert \nabla u\Vert^2_{L^2(\Omega)} + & \vert \lambda \vert \, \Vert u\Vert^2_{L^2(\Omega)} \leq C \left(\vert \lambda \vert^{-1/3}\Vert \nabla u\Vert^{2/3}_{L^2(\Omega)}\Vert h \Vert^{4/3}_{L^2(\Gamma)} + \vert \lambda \vert^{-1} \Vert h \Vert^2_{L^2(\Gamma)} \right)
\\\\
& \leq \frac{1}{2} \Vert \nabla u \Vert^2_{L^2(\Omega)} +  C \left(\vert \lambda \vert^{-1/2}\Vert h\Vert^{2}_{L^2(\Gamma)} +  \vert \lambda \vert^{-1}\Vert h\Vert^{2}_{L^2(\Gamma)}\right)
\end{array}
\end{equation*}
and then the following estimate
\begin{equation*}\label{estiminter1/43/4}
 \vert \lambda \vert^{3/4}  \, \Vert u\Vert_{L^2(\Omega)} + \vert \lambda \vert^{1/4} \Vert \nabla u\Vert_{L^2(\Omega)}  \leq C \Vert h \Vert_{L^2(\Gamma)}.
\end{equation*}
Now from the inequalities  \eqref{inegGrisvard} and  \eqref{estiminter1/43/4}, we obtain the following estimate
\begin{equation*}\label{estiminter1/43/4Gamma}
\vert \lambda \vert^{1/2} \Vert  u\Vert_{L^2(\Gamma)}  \leq C \Vert h \Vert_{L^2(\Gamma)},
\end{equation*}
where all the constants above depend only on $\Omega$ and $\omega$. So that there exists a constant $C_1$ depending only on $\Omega$ and $\omega$ such that
\begin{equation*}\label{estiminter1/43/41/2Gamma}
 \vert \lambda \vert^{3/4}  \, \Vert u\Vert_{L^2(\Omega)} + \vert \lambda \vert^{1/4} \Vert \nabla u\Vert_{L^2(\Omega)}+ \vert \lambda \vert^{1/2} \Vert  u\Vert_{L^2(\Gamma)}  \leq C_1 \Vert h \Vert_{L^2(\Gamma)}.
\end{equation*}

\noindent{\bf Step 2.} Rewrite now the relation \eqref{egaliteenergiebbis} as follow:
\begin{equation}\label{egaliteenergiebbis1}
\lambda^2 \Vert u\Vert^2_{L^2(\Omega)} = - \Vert \nabla u\Vert^2_{L^2(\Omega)} + \int_{\Gamma} \overline{u}h 
\end{equation}
and take the module of the both sides of the previous equality. Using then the estimates \eqref{estiminter1/43/4} and \eqref{estiminter1/43/4Gamma}, we get 
\begin{equation}\label{estiminter1/43/4a}
 \vert \lambda \vert^{5/4}  \, \Vert u\Vert_{L^2(\Omega)} \leq C \Vert h \Vert_{L^2(\Gamma)}.
\end{equation}
From the inequalities \eqref{estimgrad2+lamdau},  \eqref{inegGrisvard},  \eqref{estiminter1/43/4} \eqref{estiminter1/43/4a} and then Young inequality, we find
\begin{equation*}
\begin{array}{rl}
 \Vert \nabla u\Vert^2_{L^2(\Omega)} +  \vert \lambda \vert \, \Vert u\Vert^2_{L^2(\Omega)}& \leq  C \left(\vert \lambda \vert^{-5/8} \Vert \nabla u\Vert^{1/2}_{L^2(\Omega)}\Vert h \Vert^{3/2}_{L^2(\Gamma)} + \vert \lambda \vert^{-5/4} \Vert h \Vert^{2}_{L^2(\Gamma)}\right) \\\\
 & \leq C \left(\vert \lambda \vert^{-5/6} + \vert \lambda \vert^{-5/4}) \Vert h \Vert^{2}_{L^2(\Gamma)} \right) \\\\
 & \leq C\vert \lambda \vert^{-5/6}  \Vert h \Vert^{2}_{L^2(\Gamma)},
 \end{array}
\end{equation*}
since $\vert \lambda \vert \geq \omega$, and as consequence
\begin{equation}\label{estiminter1/43/4ab}
\vert \lambda \vert^{5/12} \Vert \nabla u\Vert_{L^2(\Omega)} \leq C \Vert h \Vert_{L^2(\Gamma)}.
\end{equation}
Now from the inequalities  \eqref{inegGrisvard}, \eqref{estiminter1/43/4a}  and \eqref{estiminter1/43/4ab}, we obtain the following estimate
\begin{equation*}\label{estiminter1/43/4Gamma1}
\vert \lambda \vert^{5/6} \Vert  u\Vert_{L^2(\Gamma)}  \leq C \Vert h \Vert_{L^2(\Gamma)}.
\end{equation*}
Repeating this reasoning, we find at Step $k$, with $k\geq 3$, the following estimate
\begin{equation*}\label{estiminterstepk1}
\vert \lambda \vert^{3r_{k}} \,  \Vert u\Vert_{L^2(\Omega)}  +\;   \vert \lambda \vert^{r_{k}}  \Vert  u\Vert_{H^1(\Omega)}  + \;  \vert \lambda \vert^{2r_{k}} \, \Vert u\Vert_{L^2(\Gamma)}   \leq C_{k} \Vert  h\Vert_{L^2(\Gamma)}, 
\end{equation*}
with $C_k$ depending only on $\Omega$, $\omega$ and $k$ and 
$$
r_k = \frac{2\times 3^{k-1} - 1}{4\times 3^{k-1}} < \frac{1}{2}\quad \mathrm{ and }\quad  \lim_{k\rightarrow \infty}r_{k} = \frac{1}{2}.
$$
We deduce that for any $r < 1/2$, there exists a constant $C_k$ depending only on $\Omega$, $\omega$ and $k$ such that we have the following estimate
\begin{equation}\label{estim1bbisrincomp}
\vert \lambda \vert^{3r} \,  \Vert u\Vert_{L^2(\Omega)}  +\;   \vert \lambda \vert^{r}  \Vert  u\Vert_{H^1(\Omega)}  + \;  \vert \lambda \vert^{2r} \, \Vert u\Vert_{L^2(\Gamma)}  \leq C_r \Vert  h\Vert_{L^2(\Gamma)}.
\end{equation} 
Finally, using \eqref{RellichBi} and \eqref{estim1bbisrincomp}, we get that $\Vert \nabla_{\mathscr{T}}u \Vert_{L^2(\Gamma)} \leq C(\Omega, \omega) \Vert  h\Vert_{L^2(\Gamma)}$ and then the required estimate.\end{proof}
 
\begin{remark}\label{rq33}\upshape Now using \eqref{estim1bbisr}, \eqref{energieavecd} and  \eqref{estimabe}, we can prove easily that for any $r<1/2$, there exists a constant $C_r$ depending only on $r$, $\Omega$ and $\omega$ such that we have following inequality
\begin{equation}\label{estimabdb}
\vert\lambda\vert \, \Vert \sqrt \rho \nabla u\Vert_{L^2(\Omega)} + \Vert  u\Vert_{H^{3/2}(\Omega)}  +  \Vert \sqrt \rho \nabla^2 u\Vert_{L^2(\Omega)} \leq C_r\vert\lambda\vert^{1-2r}\Vert h\Vert_{L^2(\Gamma)},
\end{equation}

\end{remark}

\begin{remark}\label{rq34}\upshape
According to (\ref{estim1bbisr}) we have that 
$$
\vert \lambda \vert^{2r} \, \Vert u\Vert_{L^2(\Gamma)}   \leq C_r \Vert  h\Vert_{L^2(\Gamma)}.
$$
Which is optimal, according to \cite{tata} and implies in particular that there exists $T >0$ such that for the corresponding solution of the wave equation:
$$
\left\{
\begin{array}{llc}
\partial_t^2 \psi - \Delta \psi = 0, \quad \mathrm{in}\;  \Omega \times (0,T), \\
\frac{\partial \psi}{\partial n} = k \in L^2(0,T;L^2_0(\Gamma)), \\
\psi(x,0) = 0,\quad \partial_t \psi (x,0) = 0, \quad \mathrm{on} \; \Omega,
\end{array}
\right.
$$
$\partial_t \psi \notin L^2(0,T; L^2(\Gamma))$, and consequently the lack of the open loop admissibility (see for details \cite{at,AN}).  
\end{remark}

Given $F\in L^2(\Omega)$, consider now the following problem:
\begin{equation}\label{PN0}
(\mathcal{P}_N^0)\ \ \ \ \lambda^2 w - \Delta w = F\quad \ \mbox{in}\ \Omega \quad
\mbox{and } \quad \frac{\partial w}{\partial\textit{\textbf n}} = 0 \ \ \mbox{on }\Gamma.
\end{equation}

\begin{theorem}\label{th2BisN} $\mathrm{(}${\bf Resolvent Interior Estimates I for} $(\mathcal{P}_N^0) \mathrm{)}$ Let $F \in L^2(\Omega)$ and suppose $\lambda $ satisfying the following condition:
\begin{equation*}\label{hyplambdaF}
[\lambda \in \mathbb{R}^*\quad  with \quad   \vert \lambda \vert \ge \lambda_0 \quad  if \; \int_{\Omega} F \not= 0] \quad or \quad [ \lambda \in\mathbb{C} \quad  with \quad   \mathrm{Re}\, \lambda \ge \omega  ], 
\end{equation*}
with $\lambda_0 > 0$  and  $\omega > 0$ arbitrary. Then Problem $(\mathcal{P}_N^0)$ has a unique solution $w$ verifying the properties \eqref{regularite}. Moreover, in the real case, we have the following estimate
\begin{equation}\label{estim1Bis}
\vert \lambda \vert^{2} \,  \Vert w\Vert_{L^2(\Omega)}   +\;    \vert \lambda \vert  \Vert  w\Vert_{H^1(\Omega)}  
 +  \vert \lambda \vert^{1/2} \Vert w\Vert_{H^{3/2}(\Omega)} \; 
 + \;     \vert \lambda \vert^{3/2}  \, \Vert w\Vert_{L^2(\Gamma)}  +\;  \vert \lambda \vert^{1/2}  \Vert w \Vert_{H^1(\Gamma)} \leq C \Vert  F\Vert_{L^2(\Omega)}.
\end{equation}
In the complex case, for any $r < 1$, there exists a constant $C_r$ depending only on $r$, $\Omega$ and $\omega$ such that we have following inequality
\begin{equation*}\label{estim1BisComp}
\begin{array}{rl}
\vert \lambda \vert^{\frac{5}{2}r} \,  \Vert  w \Vert_{L^2(\Omega)} & +\;    \vert \lambda \vert^{\frac{3}{2}r}  \Vert  w\Vert_{H^1(\Omega)}  + \;   \vert\lambda\vert^{2r- 1} \Vert  w\Vert_{H^{3/2}(\Omega)}  + \;  \vert \lambda \vert^{2r} \, \Vert w\Vert_{L^2(\Gamma)}  +  \vert \lambda \vert^{r} \Vert \nabla_{\mathscr{T}}w \Vert_{L^2(\Gamma)}   \\\\
&  \leq C_r \vert \lambda\vert^{\frac{1}{2}}\Vert  F \Vert_{L^2(\Omega)}.
\end{array}
\end{equation*}
\end{theorem}

\begin{proof} 
Clearly Problem $(\mathcal{P}_N^0)$ has a  unique solution $w\in H^{1}(\Omega)$. As $ \Delta w \in L^2(\Omega)$, then $w\in H^{3/2}(\Omega)$ and $\sqrt \rho \,  \nabla^2 w \in L^2(\Omega)$.  

We give the proof only in the case where $\int_{\Omega} F = 0$. If not, we proceed as in the proof of Theorem \ref{th1}. 
Extending $F$ by $0$ outside of $\Omega$ and using the fundamental solution $E$ of the operator $\lambda^2 I - \Delta$, the function $z = E \star \widetilde{F}$ belongs to $H^2(\R^N)$ and satisfies 
$$ 
\lambda^2 z - \Delta z = \widetilde{F}\quad \mathrm{ in }\; \R^N\quad  \mathrm{and } \quad \hat{z}(\xi)  = C_N (\lambda^2 + \vert \xi^2 \vert)^{-1} {\widehat{\widetilde{F}}}(\xi),
$$
where $\hat{z}$ is the Fourier transform of $z$. 

\noindent{\bf i) Case $\lambda$ real.} As the average of $z$ in $\Omega$ is equal to zero, multiplying the above equation by $z$ and then by $- \Delta z$ the above equation and integrating in $\R^N$, we get the following estimates 
\begin{equation}\label{estim102}
\mathrm{Max}\{\vert\lambda\vert^2 \Vert z\Vert_{L^2(\R^N)},\, \vert\lambda\vert\,  \Vert \nabla z\Vert_{L^2(\R^N)},\, \frac{1}{2}\Vert \Delta z\Vert_{L^2(\R^N)}\} \leq  \Vert F\Vert_{L^2(\Omega)} .
\end{equation}
Note that the previous estimate can be obtained by using directly the explicit expression of the solution $z$, or by interpolation thanks to the estimate \eqref{estim102}, from which we deduce that
\begin{equation}\label{estim103}
\vert \lambda \vert^{1/2}  \Vert  z\Vert_{H^{3/2}(\R^N)}\leq C\Vert F\Vert_{L^2(\Omega)}.
\end{equation}

Recall now that if $\varphi \in H^1(\Omega)$, for any $\varepsilon > 0$, we have
\begin{equation}\label{normL2Gamma}
\int_{\Gamma} \vert \varphi  \vert^2 \leq C (\Omega) \left(\varepsilon^{-1} \Vert \varphi  \Vert^2_{L^2(\Omega)} + \varepsilon  \Vert \varphi  \Vert^2_{H^1(\Omega)}\right).
\end{equation}
Choosing $\varepsilon = \vert \lambda \vert^{-1}$ and successively  $\varphi = z$ and then $ \varphi = \nabla z$, we obtain the following inequality
\begin{equation}\label{estim104}
\vert \lambda \vert^{3/2}  \Vert  z\Vert_{L^{2}(\Gamma)} + \vert \lambda \vert^{1/2}  \Vert  \nabla z\Vert_{L^{2}(\Gamma)}\leq C(\Omega) \Vert F\Vert_{L^2(\Omega)}.
\end{equation}

Setting then $v = z_{\vert\Omega} - w$, we have  $\lambda^2 v - \Delta v = 0$ in $\Omega$ and $\frac{\partial v}{\partial\textit{\textbf n}} =  \frac{\partial z}{\partial\textit{\textbf n}}$ on $\Gamma$. Using Theorem \ref{th1}, \eqref{estim102} and \eqref{estim104} we get
\begin{equation}\label{estim105}
\begin{array}{rl}
\vert \lambda \vert^{3/2} \,  \Vert v\Vert_{L^{2}(\Omega)}   &  + \;  \vert \lambda \vert^{1/2} \Vert  v\Vert_{H^1(\Omega)}   + \Vert v\Vert_{H^{3/2}(\Omega)}  + \;  \vert \lambda \vert  \, \Vert v\Vert_{L^2(\Gamma)} + \; \Vert v \Vert_{H^1(\Gamma)  } \\\\  
& \leq C \Vert \frac{\partial z}{\partial\textit{\textbf n}} \Vert_{L^2(\Gamma)}\    \leq C   \vert \lambda \vert^{-1/2}\Vert  F\Vert_{L^2(\Omega)}.
\end{array}
\end{equation}
The required estimate is finally a consequence of \eqref{estim102}, \eqref{estim103}, \eqref{estim104} and \eqref{estim105}. \medskip

\noindent{\bf ii) Case $\lambda$ complex.} Multiplying by $\overline{z}$, we have the relation
\begin{equation*}\label{egaliteenergiebbiscomp}
\lambda^2 \Vert z\Vert^2_{L^2(\R^N)} + \Vert \nabla z\Vert^2_{L^2(\R^N)} = \int_{\Omega} F\overline{z},
\end{equation*}
from which we deduce as in Step 1 of the proof of Theorem \ref{th31} that
\begin{equation}\label{inegaliteenergiebbiscompa}
\vert \lambda\vert \Vert z\Vert_{L^2(\R^N)} + \vert \lambda\vert^{1/2} \Vert \nabla z\Vert_{L^2(\R^N)} \leq C \Vert  F \Vert_{L^2(\Omega)}.
\end{equation}
Writing the above equality as follows
$$
\lambda^2 \Vert z\Vert^2_{L^2(\R^N)} = - \Vert \nabla z\Vert^2_{L^2(\R^N)} + \int_{\Omega} F\overline{z},
$$
and taking the module of the both sides of the previous equality, we get thanks to \eqref{inegaliteenergiebbiscompa} the following estimate 
\begin{equation*}\label{estiminter3/23/4acomp}
 \vert \lambda\vert^{3/2} \Vert z\Vert_{L^2(\R^N)} + \vert \lambda\vert^{3/4} \Vert \nabla z\Vert_{L^2(\R^N)} \leq C \Vert  F \Vert_{L^2(\Omega)}.
\end{equation*}
Repeating this reasoning, we find at Step $k$, with $k\geq 2$, the following estimate
\begin{equation*}\label{estiminterstepk2}
 \vert \lambda\vert^{r_k} \Vert z\Vert_{L^2(\R^N)} + \vert \lambda\vert^{\frac{1}{2}r_k} \Vert \nabla z\Vert_{L^2(\R^N)} \leq C \Vert  F \Vert_{L^2(\Omega)}
\end{equation*}
with 
$$
r_k = 1 +  \frac{1}{2}r_{k-1}, \; k \ge 1\quad\mathrm{ and  } \quad  r_0 = 0.
$$
We verify easily that for any $k \geq 1$, we have $r_k < 2$ and $\ds \lim_{k\rightarrow \infty}r_{k} = 2.$ We deduce that for any $r < 1$, there exists a constant $C_r$ depending only on $\Omega$, $\omega$ and $r$ such that we have the following estimate
\begin{equation}\label{estim2bbisrincompaa}
 \vert \lambda\vert^{2r} \Vert z\Vert_{L^2(\R^N)} + \vert \lambda\vert^{r} \Vert \nabla z\Vert_{L^2(\R^N)} \leq C_r \Vert  F \Vert_{L^2(\Omega)}.
\end{equation} 
As $\Delta z = \lambda^2 z$, we get
\begin{equation*}\label{estim2bbisDeltarincomp}
 \vert \lambda\vert^{2r} \Vert z\Vert_{L^2(\R^N)} + \vert \lambda\vert^{r} \Vert \nabla z\Vert_{L^2(\R^N)} + \vert \lambda\vert^{2(r-1)} \Vert \Delta z\Vert_{L^2(\R^N)} \leq C_r \Vert  F \Vert_{L^2(\Omega)}.
\end{equation*} 
So by interpolation, 
\begin{equation*}\label{estimH3demi}
 \vert \lambda\vert^{2r - \frac{3}{2}} \Vert z\Vert_{H^{3/2}(\R^N)} \leq C_r \Vert  F \Vert_{L^2(\Omega)}.
\end{equation*} 
As in Point i) we deduce from \eqref{normL2Gamma} and \eqref{estim2bbisrincompaa} that
\begin{equation*}
\Vert z\Vert^2_{L^2(\Gamma)} \leq C( \vert \lambda\vert \vert \lambda\vert^{-4r} +  \vert \lambda\vert^{-1} \vert \lambda\vert^{-2r})\Vert  F \Vert^2_{L^2(\Omega)}   \leq C \vert \lambda\vert^{1-4r} \Vert  F \Vert^2_{L^2(\Omega)}.
\end{equation*} 
Proceeding by the same way for the estimate of $\Vert \nabla z\Vert^2_{L^2(\Gamma)} $, we get finally
\begin{equation}\label{estisurzL2GH2O}
\vert \lambda\vert^{2r-\frac{1}{2}} \Vert z\Vert_{L^2(\Gamma)} + \vert \lambda\vert^{r-\frac{1}{2}} \Vert \nabla z\Vert_{L^2(\Gamma)}     \leq C \vert  \Vert  F \Vert_{L^2(\Omega)}.
\end{equation} 
Setting as above  $v = z_{\vert\Omega} - w$ and using Theorem \ref{th31}, Remark \ref{rq33} and \eqref{estisurzL2GH2O} we get
\begin{equation*}\label{estim105a}
\begin{array}{rl}
\vert \lambda \vert^{\frac{3}{2}r} \,  \Vert  v \Vert_{L^2(\Omega)} & +\;   \vert \lambda \vert^{\frac{1}{2}r}  \Vert  v\Vert_{H^1(\Omega)}   + \;    \vert\lambda\vert^{r- 1} \Vert  v\Vert_{H^{3/2}(\Omega)} + \;  \vert \lambda \vert^{r} \, \Vert v\Vert_{L^2(\Gamma)}  + \Vert \nabla_{\mathscr{T}}v \Vert_{L^2(\Gamma)} \\\\
& \leq C_r\Vert \frac{\partial z}{\partial\textit{\textbf n}}\Vert_{L^2(\Gamma)} \leq C_r  \vert \lambda\vert^{\frac{1}{2}-r} \Vert  F \Vert_{L^2(\Omega)}.
\end{array}
\end{equation*}
So
\begin{equation}\label{estim105aa}
\begin{array}{rl}
\vert \lambda \vert^{\frac{5}{2}r} \,  \Vert  v \Vert_{L^2(\Omega)} & +\;    \vert \lambda \vert^{\frac{3}{2}r}  \Vert  v\Vert_{H^1(\Omega)}  + \;  \vert \lambda \vert^{2r} \, \Vert v\Vert_{L^2(\Gamma)}  +  \vert \lambda \vert^{r} \Vert \nabla_{\mathscr{T}}v \Vert_{L^2(\Gamma)}  + \;  \vert\lambda\vert^{2r- 1} \Vert  v\Vert_{H^{3/2}(\Omega)}  \\\\
&  \leq C_r \vert \lambda\vert^{\frac{1}{2}}\Vert  F \Vert_{L^2(\Omega)}.
\end{array}
\end{equation}
The required estimate of $w$, which is actually obtained from that of $v$, is finally a consequence of \eqref{estim105aa} and the above estimates of $z$.\end{proof}
 \medskip
 
 \begin{remark}\label{H1demi} \upshape Using an interpolate argument we deduce immediately that the solution of the above theorem satisfies the following estimates 
 \begin{equation*}\label{estim10-}
  \vert \lambda \vert^{3/2} \Vert u\Vert_{H^{1/2}(\Omega)} \leq C \Vert  F \Vert_{L^2(\Omega)},\quad \mathrm{ if}\;  \lambda \quad \mathrm{is\,\,  real }
\end{equation*}
 and for anr $r < 1$
\begin{equation*}\label{estim10-r}
  \vert \lambda \vert^{2r-1/2} \Vert u\Vert_{H^{1/2}(\Omega)} \leq C_r  \Vert  F \Vert_{L^2(\Omega)}, \quad \mathrm{ if}\;  \lambda \quad \mathrm{is\,\,  complex}.
\end{equation*}
 \end{remark}

Recall now that the following lemma (see \cite{AMN}).

\begin{lemma}\label{traces+Green} Let $\Omega$ be a  bounded open subset of class $ \mathcal{C}^{1,1}$ of $\R^N$. Then the space $\mathcal{D}(\overline{\Omega})$ is dense in the space
$$
E(\Omega, \Delta) = \{\varphi \in L^2(\Omega); \; \Delta \varphi \in L^2(\Omega)\}.
$$
Moreover, the linear mapping $\varphi\mapsto (\varphi_{\vert\Gamma},  \frac{\partial \varphi }{\partial\textit{\textbf n}})$ defined on  $\mathcal{D}(\overline{\Omega})$ can be continuously extended to a linear and continuous mapping from $E(\Omega, \Delta)$ to $H^{-1/2}(\Gamma)\times H^{-3/2}(\Gamma)$ and we have the following Green formula: for any $\varphi \in E(\Omega, \Delta)$ and $v\in H^2(\Omega)$
$$
\int_\Omega \varphi \Delta v - \int_\Omega v\Delta \varphi  = \langle \varphi , \, \frac{\partial v}{\partial\textit{\textbf n}}\rangle_{H^{-1/2}(\Gamma)\times H^{1/2}(\Gamma)} - 
\langle  \frac{\partial \varphi }{\partial\textit{\textbf n}}, \,\ v \rangle_{H^{-3/2}(\Gamma)\times H^{3/2}(\Gamma)}.
$$

\end{lemma}
 
 \begin{theorem}\label{th2Bi} $\mathrm{(}${\bf Very Weak Estimate for} $(\mathcal{P}_N^H) \mathrm{)}$ Let $h\in H^{-1}(\Gamma)$ and suppose $\lambda $ satisfying the following condition:
\begin{equation}
\label{hyplambdah}
[\lambda \in \mathbb{R}^*\quad \mathrm{ with }\quad   \vert \lambda \vert \ge \lambda_0 \quad  \mathrm{ if }\; \langle h, 1\rangle \not= 0] \quad \mathrm{ or } \quad [ \lambda \in\mathbb{C} \quad \mathrm{ with}\quad   \mathrm{Re}\, \lambda \ge \omega  ], 
\end{equation}
with $\lambda_0 > 0$  and  $\omega > 0$ arbitrary. Then Problem $(\mathcal{P}_N^H)$ has a unique solution which satisfies 
\begin{equation*}
u\in H^{1/2}(\Omega),\quad u_{\vert\Gamma}\in L^2(\Gamma)
\end{equation*}
with the following estimates:
\begin{equation}\label{estim9}
  \vert \lambda \vert^{1/2} \Vert u\Vert_{L^2(\Omega)} +  \Vert u\Vert_{H^{1/2}(\Omega)}  + \Vert u\Vert_{L^2(\Gamma)} \leq C \Vert  h\Vert_{H^{-1}(\Gamma)},\quad \mathrm{ if}\;  \lambda \; \mathrm{is\,\,  real }
\end{equation}
and for any $r < 1$
\begin{equation}\label{estim9r}
  \vert \lambda \vert^{r-1/2} \Vert u\Vert_{L^2(\Omega)} +  \vert \lambda \vert^{3(r- 1)}\Vert  \Vert u\Vert_{H^{1/2}(\Omega)}  + \Vert u\Vert_{L^2(\Gamma)}  \leq C_r \Vert  h\Vert_{H^{-1}(\Gamma)}, \quad \mathrm{ if}\;  \lambda \; \mathrm{is\,\,  complex}.
\end{equation}
where $C$ and $C_r$ are not depending on $\lambda$.
\end{theorem}  

\begin{proof} {\bf Step 1.} {\it We suppose that $\Omega$ is $\mathcal{C}^{1,1}$}. 

Firstly, observe that if $u\in L^2(\Omega)$ solves Problem $(\mathcal{P}_N^H)$, then $\Delta u \in L^2(\Omega)$ and consequently $ u_{\vert\Gamma}\in H^{-1/2}(\Gamma)$ and $\frac{\partial u}{\partial\textit{\textbf n}}\in  H^{-3/2}(\Gamma)$ since $\Omega$ is $\mathcal{C}^{1,1}$ (see \cite{ARB}).  
But, by hypothesis, the normal derivative of $u$ belongs to $ H^{-1}(\Gamma)$. So we will prove that $u\in L^2(\Gamma)$ with the estimate \eqref{estim9}. Moreover, to find  $u\in L^{2}(\Omega)$ solution of  problem $(\mathcal{P}_N^H)$ is in fact equivalent to solve the following very weak formulation: For any $w\in H^1(\Omega)$ such that $ \Delta w \in L^2(\Omega)$ and  $\frac{\partial w}{\partial\textit{\textbf n}} = 0$ on $\Gamma,$ \begin{equation}\label{vwf}
\int_{\Omega} u(\lambda^2 w - \Delta w) = \langle h, \, w \rangle_{H^{-1}(\Gamma)\times H^{1}(\Gamma)}.
\end{equation}
Observe that  with the assumptions concerning $w$, from Ne$\mathrm{\check{c}}$as property we get $w_{\vert\Gamma}\in H^1(\Gamma)$, which gives a sense for the above duality bracket. Choosing then the test function $w$ in $\mathcal{D}(\Omega)$, we obtain 
$$ 
\lambda^2 u - \Delta u = 0\quad \mathrm{ in}\;  \Omega .
$$
From which, after using Lemma \ref{traces+Green}, we deduce that for any $w\in H^2(\Omega)$ such that $\frac{\partial w}{\partial\textit{\textbf n}} = 0$ on $\Gamma$:
\begin{equation*}
\ds \langle \frac{\partial u}{\partial\textit{\textbf n}}, \, w \rangle_{H^{-3/2}(\Gamma)\times H^{3/2}(\Gamma)} = \langle h, \, w \rangle_{H^{-1}(\Gamma)\times H^{1}(\Gamma)}
\end{equation*}
and then 
\begin{equation*}
\frac{\partial u}{\partial\textit{\textbf n}} = h \quad \mathrm{on } \; \Gamma.
\end{equation*}
\medskip
 {\bf Step 2.}  {\it We suppose again that $\Omega$ is $\mathcal{C}^{1,1}$ and we will prove the interior estimate.}
 
To solve our problem remains then to ensure the existence of a unique $u\in L^2(\Omega)$ satisfying the above very weak formulation \eqref{vwf}. But using Theorem \ref{th2BisN}, for any $F\in L^2(\Omega)$, there exists a unique $w\in H^{3/2}(\Omega)$ solution of Problem \eqref{PN0}. Hence by \eqref{estim1Bis}
\begin{equation*} 
\left|  \int_{\Omega}u F\right| \leq \Vert h \Vert_{H^{-1}(\Gamma)}\Vert w \Vert_{H^{1}(\Gamma)} \leq C  \vert \lambda \vert^{-1/2}  \Vert h \Vert_{H^{-1}(\Gamma)}\Vert F \Vert_{L^{2}(\Omega)},
\end{equation*}
if $\lambda$ is real and by \eqref{estim1BisComp}, when $\lambda$ is complex, for any $r < 1$ 
\begin{equation*} 
\left|  \int_{\Omega}u F\right| \leq \Vert h \Vert_{H^{-1}(\Gamma)}\Vert w \Vert_{H^{1}(\Gamma)} \leq C_r  \vert \lambda \vert^{1/2-r}  \Vert h \Vert_{H^{-1}(\Gamma)}\Vert F \Vert_{L^{2}(\Omega)}.
\end{equation*}
Above, the constants $C$ and $C_r$ depend only on the Lipschitz character of $\Omega$. Moreover $C$ depends also on $\lambda_0$ when $  \langle h, 1\rangle \not= 0$ and $C_r$ depends also on $\omega$ when $\lambda$ is a complex number. In other words, the linear form
$$
T \; :\;  F \mapsto \langle h, \, w \rangle_{H^{-1}(\Gamma)\times H^{1}(\Gamma)}
$$
is continuous on $L^2(\Omega)$ and hence there exists a unique  $u\in L^2(\Omega)$ such that for any $F \in L^{2}(\Omega)$, 
$$
T(F) =  \int_{\Omega}u F
$$
which means that $u$ is solution of \eqref{vwf}. Moreover we have the following estimate 
\begin{equation}\label{estim10}
  \vert \lambda \vert^{1/2} \Vert u\Vert_{L^2(\Omega)} \leq C \Vert  h\Vert_{H^{-1}(\Gamma)},\quad \mathrm{ if}\;  \lambda \quad \mathrm{is\,\,  real }
\end{equation}
 and 
\begin{equation}\label{estim10r}
  \vert \lambda \vert^{r-1/2} \Vert u\Vert_{L^2(\Omega)} \leq C_r \Vert  h\Vert_{H^{-1}(\Gamma)}, \quad \mathrm{ if}\;  \lambda \quad \mathrm{is\,\,  complex}.
\end{equation}

Now, we will prove that $u\in H^{1/2}(\Omega)$ with the corresponding estimate. For this, we distinguish two cases.

\noindent {\bf i)} Suppose firstly that $\langle h, 1\rangle = 0$. So that $\int_\Omega u = 0$. We write now $u = v + y$ as  sum of solutions of the following problems:
\begin{equation*}
 \Delta v =   \lambda^2 u \quad \ \mbox{in}\ \Omega \quad
\mbox{and } \quad \frac{\partial v}{\partial\textit{\textbf n}} = 0 \ \ \mbox{on }\Gamma 
\end{equation*}
and
\begin{equation*}
 \Delta y =  0\quad \ \mbox{in}\ \Omega \quad
\mbox{and } \quad \frac{\partial y}{\partial\textit{\textbf n}} = h \ \ \mbox{on }\Gamma ,
\end{equation*}
where on one hand $v\in H^s(\Omega)\cap L^2_0(\Omega)$ for any $s < 3/2$ and  $y\in H^{1/2}(\Omega)\cap L^2_0(\Omega)$. On the other hand, using Theorem \ref{th2BisN}, Remark \ref{H1demi}, \eqref{estim10} and \eqref{estim10r} the solution $v$ satisfies the following estimate:
\begin{equation*}
  \vert \lambda \vert^{3/2} \Vert v\Vert_{H^{1/2}(\Omega)}  \leq C  \vert \lambda \vert^{2} \Vert u\Vert_{L^2(\Omega)} \leq C  \vert \lambda \vert^{3/2}\Vert  h\Vert_{H^{-1}(\Gamma)},\quad \mathrm{ if}\;  \lambda \; \mathrm{is\,\,  real }
\end{equation*}
 and for any $r < 1$
\begin{equation*}
  \vert \lambda \vert^{2r-1/2} \Vert v\Vert_{H^{1/2}(\Omega)} \leq C_r  \vert \lambda \vert^{2} \Vert u\Vert_{L^2(\Omega)}  \leq C_r  \vert \lambda \vert^{\frac{5}{2} - r }\Vert  h\Vert_{H^{-1}(\Gamma)}, \quad \mathrm{ if}\;  \lambda \; \mathrm{is\,\,  complex}.
\end{equation*}
So that we get
\begin{equation*}\label{estim10a}
\Vert v\Vert_{H^{1/2}(\Omega)}   \leq C \Vert  h\Vert_{H^{-1}(\Gamma)},\quad \mathrm{ if}\;  \lambda \; \mathrm{is\,  real }
\end{equation*}
and for any $r < 1$
\begin{equation*}\label{estim10ar}
\Vert v\Vert_{H^{1/2}(\Omega)}  \leq C_r  \vert \lambda \vert^{3(1- r )}\Vert  h\Vert_{H^{-1}(\Gamma)}, \quad \mathrm{ if}\;  \lambda \; \mathrm{is\,  complex},
\end{equation*}
where  the constants $C$ and $C_r$ depend only on the Lipschitz character of $\Omega$. 

Concerning the estimate of the solution $y$, from Theorem \ref{PNH, lambda=0}  we know that
\begin{equation*}\label{estim10ay}
\Vert y\Vert_{H^{1/2}(\Omega)}   \leq C \Vert  h\Vert_{H^{-1}(\Gamma)},\quad \mathrm{ if}\;  \lambda \; \mathrm{is\,\,  real \, or \, 
complex}.
\end{equation*}

\noindent{\bf ii)} Suppose now that $\langle h, 1\rangle \not= 0$ and setting $\widetilde{h} = h - \frac{1}{\vert \Gamma\vert}\langle h, 1\rangle$. The solution $u$ is then a sum of $\widetilde{u}$, solution of  $(\mathcal{P}_N^H)$ with the Neumann boundary condition $\widetilde{h}$ and of $a$,  solution of  $(\mathcal{P}_N^H)$ with the Neumann boundary condition $\frac{1}{\vert \Gamma\vert}\langle h, 1\rangle$. We finally prove easily the corresponding estimates  for these two solutions.

\medskip

\noindent{\bf Step 3.}  {\it We suppose always that $\Omega$ is $\mathcal{C}^{1,1}$ and we will prove the boundary estimate.}

For that, let $f\in H^{1/2}(\Gamma)$ and $z\in H^{2}(\Omega)$ the unique solution verifying 
\begin{equation*}\label{PNHbis}
 \lambda^2 z - \Delta z = 0\quad \ \mbox{in}\ \Omega \quad
\mbox{and } \quad \frac{\partial z}{\partial\textit{\textbf n}} = f \ \ \mbox{on }\Gamma .
\end{equation*}
From Theorem  \ref{th1}, Theorem \ref{th31} and Remark \ref{rq33}, the solution $z$ satisfies the estimate
\begin{equation}\label{estim11}
  \Vert \lambda z\Vert_{L^2(\Gamma)} + \Vert  z \Vert_{H^1(\Gamma)} +    \vert \lambda \vert^{1/2} \Vert  z \Vert_{H^1(\Omega)}  +  \Vert  z \Vert_{H^{3/2}(\Omega)} \leq C \Vert  f\Vert_{L^2(\Gamma)} 
\end{equation}
if $\lambda$ is real and or any $r < 1$, 
\begin{equation}\label{estim11r}
 \vert \lambda\vert^{r} \Vert z\Vert_{L^2(\Gamma)} + \Vert  z \Vert_{H^1(\Gamma)} +    \vert \lambda \vert^{r/2} \Vert  z \Vert_{H^1(\Omega)}  +  \vert \lambda \vert^{r-1} \Vert  z \Vert_{H^{3/2}(\Omega)} \leq C_r \Vert  f\Vert_{L^2(\Gamma)}, 
\end{equation}
if $\lambda$ is complex, where the constant $C$ and $C_r$ depend as in Step 2 only on the Lipschitz character of $\Omega$. Thanks to Lemma \ref{traces+Green}, we have
$$
\langle u ,\,  \frac{\partial z}{\partial\textit{\textbf n}}\rangle_{H^{-1/2}(\Gamma)\times H^{1/2}(\Gamma)} =  \langle h, \, z \rangle_{H^{-1}(\Gamma)\times H^{1}(\Gamma)},
$$
and then from \eqref{estim11} when $\lambda$ is real:
$$
\left| \langle u ,\,  f\rangle_{H^{-1/2}(\Gamma)\times H^{1/2}(\Gamma)} \right| \leq  C\Vert h \Vert_{H^{-1}(\Gamma)}\Vert f \Vert_{L^2(\Gamma)}.
$$
By \eqref{estim11r}, we give the same inequality in the complex case with $C$ replaced by $C_r$. So by density of $H^{1/2}(\Gamma)$ in $L^{2}(\Gamma)$, we  deduce that the previous estimate holds for any $f\in L^2(\Gamma)$ and then $u\in L^2(\Gamma)$ with the following estimate
\begin{equation}\label{estimudansL2Gamma-VW}
\Vert  u \Vert_{L^2(\Gamma)} \leq C \Vert h \Vert_{H^{-1}(\Gamma)} \quad \mathrm{if}\; \lambda \in \R \quad  \mathrm{and}\quad \Vert  u \Vert_{L^2(\Gamma)} \leq C_r \Vert h \Vert_{H^{-1}(\Gamma)}\quad \mathrm{if}\; \lambda \in \mathbb{C},
\end{equation}
with $\lambda$ satisfying \eqref{hyplambdah}.
\medskip

\noindent {\bf Step 4.}  As in the proof of Theorem \ref{th1}, we  can then extend the estimates \eqref{estim10}, \eqref{estim10r} and \eqref{estimudansL2Gamma-VW} to the case where $\Omega$ is only Lipschitz. \end{proof}

\begin{remark} \upshape The solution $u$ given in the previous theorem satisfies probably that $\sqrt \rho\,  \nabla u\in L^2(\Omega)$. But we dont give here the proof of this result.
\end{remark}

\begin{corollary} Let $h\in H^{-s}(\Gamma)$, with $0 < s < 1$, and suppose $\lambda $ satisfying the condition \eqref{hyplambdah}.
Then Problem $(\mathcal{P}_N^H)$ has a unique solution which satisfies 
\begin{equation}\label{existH3/2(1-s)}
u\in H^{\frac{3}{2}(1-s)}(\Omega)\quad and\quad u_{\vert\Gamma}\in H^{1-s}(\Gamma)
\end{equation}
with the following estimates:
\begin{equation*}\label{estim9Hs}
\begin{array}{rl}
  \vert \lambda \vert^{\frac{3}{2}- s} \Vert u\Vert_{L^2(\Omega)} & + \; \;  \vert \lambda \vert^{\frac{1}{2}} \Vert u\Vert_{H^{1-s}(\Omega)} +  \vert \lambda \vert^{\frac{s}{2}} \Vert u\Vert_{H^{\frac{3}{2}(1-s)}(\Omega)} +  \Vert u\Vert_{H^{\frac{3}{2}-s}(\Omega)} + \\\\
  & + \;  
\vert \lambda \vert^{1-s}   \Vert u\Vert_{L^2(\Gamma)} +  \Vert u\Vert_{H^{1-s}(\Gamma)} \leq C \Vert  h\Vert_{H^{-s}(\Gamma)},
\end{array}
\end{equation*}
 if $ \lambda$ is real and for any $r < 1$
\begin{equation*}\label{estim9rHs}
\begin{array}{rl}
 & \vert \lambda \vert^{\frac{1}{2}(r(3-s)-s)} \Vert u\Vert_{L^2(\Omega)}  + \vert \lambda \vert^{\frac{1}{2}(r(1+s)-s)} \Vert u\Vert_{H^{1-s}(\Omega)}  
   + \vert \lambda \vert^{r(2s-1)-\frac{3s}{2} + 1} \Vert u\Vert_{H^{\frac{3}{2}(1-s)}(\Omega)}+ \\\\ 
   & + \; \vert \lambda \vert^{r(4s-1)- 4s + 1} \Vert u\Vert_{H^{\frac{3}{2} -s}(\Omega)} + \; \vert \lambda \vert^{r(1-s)}   \Vert u\Vert_{L^2(\Gamma)}\;  + \;  \Vert u\Vert_{H^{1-s}(\Gamma)} \leq C_r \Vert  h\Vert_{H^{-s}(\Gamma)}.
  \end{array}
\end{equation*}
 if $ \lambda$ is complex. As above the constants $C$ and $C_r$ are not depending on $\lambda$.
\end{corollary}

\begin{proof} To prove the existence of the solution $u$ satisfying \eqref{existH3/2(1-s)}, it suffices to approach the distribution $h$ by a sequence $(h_k)$ of functions belonging to  $L^2(\Gamma)$,  to use the estimates of the corresponding solutions $u_k$ obtained by interpolation thanks to the estimates \eqref{estim9}, \eqref{estim1}   when $\lambda$ is real and  \eqref{estim9r}, \eqref{estim1bbisr} and \eqref{estimabdb} when $\lambda$ is complex and finally to pass to the limite. \end{proof}

%
%
%
%
%
%
%
%
%

\subsection{Dirichlet Case}

Given $F\in L^2(\Omega)$, consider now the following Dirichlet problem:
\begin{equation*}\label{PD0}
(\mathcal{P}_D^0)\ \ \ \ \lambda^2 w - \Delta w = F\quad \ \mbox{in}\ \Omega \quad
\mbox{and } \quad w = 0 \ \ \mbox{on }\Gamma.
\end{equation*}
Let us introduce the following space
\begin{equation*}
\mathscr{T}^2_{-3/2}(\Omega) =  \left\{ v\in H^{3/2}_0(\Omega);\;  \rho^{1/2}\nabla^2 v \in \L^2(\Omega)\right\},
\end{equation*}
which is an Hilbert space for his graph norm and the following lemma (see \cite{AMN}).  

\begin{lemma} The following inequality holds:
\begin{equation*}\label{inegT2-3/2bis}
\forall v\in \mathscr{T}^2_{-3/2}(\Omega), \quad \Vert v \Vert_{H^{3/2}_{0}(\Omega)} + \Vert \rho^{1/2} \nabla^2 v   \Vert_{L^2(\Omega)} \leq  C\Vert \rho^{1/2} \Delta v  \Vert_{L^2(\Omega)}.
\end{equation*}
\end{lemma}

Using Theorem \ref{th2}, we can prove with the same ideas that for Theorem \ref{th2BisN} the following result.

\begin{theorem}\label{th3BisD} $\mathrm{(}${\bf Resolvent Interior Estimates I for} $(\mathcal{P}_D^0) \mathrm{)}$ Let $F \in L^2(\Omega)$ and suppose $\lambda $ satisfying the following condition:
\begin{equation*}\label{hyplambdaFD}
\lambda \in \mathbb{R}^*\quad  or  \quad  \lambda \in\mathbb{C} \quad  with\quad   \mathrm{Re}\, \lambda \ge \omega, 
\end{equation*}
with arbitrary $\omega > 0$. Then Problem $(\mathcal{P}_D^0)$ has a unique solution  $w\in H^{3/2}_0(\Omega)$, with $\sqrt \rho  \nabla^2 w\in L^2(\Omega), \frac{\partial w}{\partial\textit{\textbf n}}\in L^2(\Gamma)$ and satisfying the following estimate
\begin{equation}\label{estim1BisD}
\begin{array}{rl}
\vert \lambda \vert^{2} \,  \Vert w\Vert_{L^2(\Omega)} & +  \;  \vert \lambda \vert^{3/2} \;  \Vert \sqrt \rho  \nabla w\Vert_{L^2(\Omega)}  +\;   \vert \lambda \vert  \Vert  w\Vert_{H^1(\Omega)} + \vert \lambda \vert^{1/2} \Vert w\Vert_{H^{3/2}(\Omega)} + \\\\
&   +  \;   \vert \lambda \vert^{1/2}\Vert \sqrt \rho  \nabla^2 w\Vert_{L^2(\Omega)}  + \; \vert \lambda \vert^{1/2} \, \Vert \frac{\partial w}{\partial\textit{\textbf n}}\Vert_{L^2(\Gamma)}   \leq C \Vert  F\Vert_{L^2(\Omega)}
\end{array}
\end{equation}
if $\lambda$ is real. 
In the complex case, for any $r < 1$, there exists a constant $C_r$ depending only on $r$, $\Omega$ and $\omega$ such that we have following inequality
\begin{equation}\label{estim1BisCompD}
\begin{array}{rl}
\vert \lambda \vert^{3r-1} \,  \Vert  w \Vert_{L^2(\Omega)} & + \; \;  \vert \lambda \vert^{2r-1}  \Vert  w\Vert_{H^1(\Omega)}  + \;   \vert\lambda\vert^{\frac{5}{2}r-2} \Vert  w\Vert_{H^{3/2}(\Omega)} \; + \\\\ 
& \;  + \vert\lambda\vert^{\frac{5}{2}r-2}\Vert \sqrt \rho  \nabla^2 w\Vert_{L^2(\Omega)} + \vert \lambda \vert^{\frac{3}{2}r-1}\Vert \frac{\partial w}{\partial\textit{\textbf n}}\Vert_{L^2(\Gamma)}   \leq C_r \Vert  F \Vert_{L^2(\Omega)}.
\end{array}
\end{equation}
\end{theorem}

\begin{proof}

Clearly Problem $(\mathcal{P}_D^0)$ has a  unique solution $w\in H^{1}_0(\Omega)$. As $ \Delta w \in L^2(\Omega)$, then $w\in H^{3/2}(\Omega)$ and $\sqrt \rho \,  \nabla^2 w \in L^2(\Omega)$.  

\noindent{\bf i) Case $\lambda$ real.}  Multiplying the equation $\lambda^2 w - \Delta w = F$ successively by $w$ and  $-\Delta w$, we get
$$
\mathrm{Max}\{\vert\lambda\vert^2 \Vert w\Vert_{L^2(\Omega)},\, \vert\lambda\vert \Vert \nabla w\Vert_{L^2(\Omega)},\, \Vert \Delta w\Vert_{L^2(\Omega)}\} \leq  \Vert F\Vert_{L^2(\Omega)}.
$$
Using Rellich equality \eqref{Rellich} we deduce that
\begin{equation*}\label{estidernormFL2}
\vert \lambda \vert^{1/2} \, \left\|\frac{\partial w}{\partial\textit{\textbf n}}\right\|_{L^2(\Gamma)}   \leq C \Vert  F\Vert_{L^2(\Omega)}.
\end{equation*}
As in the proof of Theorem \ref{th2BisN}, we extending $F$ by $0$ outside of $\Omega$ and we consider the solution $z $ belonging to $H^2(\R^N)$ and satisfying $\lambda^2 z - \Delta z = \widetilde{F}\quad \mathrm{ in }\; \R^N$with the estimates \eqref{estim102}, \eqref{estim103} and \eqref{estim104}. Setting then  $v = z_{\vert\Omega} - w$, we have  $\lambda^2 v - \Delta v = 0$ in $\Omega$ and $v = z$ on $\Gamma$. So by Theorem \ref{th2} and \eqref{estim3} we have $v\in H^{3/2}(\Omega),\quad \sqrt \rho \,  \nabla^2 v \in L^2(\Omega), \quad \frac{\partial v}{\partial\textit{\textbf n}}\in L^2(\Gamma).$ As a consequence the function $w$ satisfies the same properties and moreover
\begin{equation*}\label{estim3bb}
\begin{array}{rl}
\vert \lambda \vert^{3/2} \,  \Vert v\Vert_{L^2(\Omega)} & + \; \;   \vert \lambda \vert^{1/2}  \Vert  u\Vert_{H^1(\Omega)} + \Vert v\Vert_{H^{3/2}(\Omega)} + \; \Vert \sqrt \rho  \nabla^2 v\Vert_{L^2(\Omega)} + \left\| \frac{\partial v}{\partial \textbf{\textit n}}\right\|_{L^2(\Gamma)} \\\\
&   \leq C(\Omega) \left( \vert \lambda \vert \; \Vert  z \Vert_{L^2(\Gamma)}  + \Vert z \Vert_{H^1(\Gamma)} \right)\leq C  \vert \lambda \vert^{-1/2}  \Vert  F\Vert_{L^2(\Omega)}.
\end{array}
\end{equation*}
With the above estimates on $z$ we find the required estimate \eqref{estim1BisD}.

\noindent{\bf ii) Case $\lambda$ complex.} Multiplying by $\overline{w}$, we have the relation
\begin{equation*}\label{egaliteenergiebbiscompvar}
\lambda^2 \Vert w\Vert^2_{L^2(\Omega)} + \Vert \nabla w\Vert^2_{L^2(\Omega)} = \int_{\Omega} F\overline{w},
\end{equation*}
from which we deduce as in Step 1 of the proof of Theorem \ref{th31} that
\begin{equation*}
\vert \lambda\vert \Vert w\Vert^2_{L^2(\Omega)} +\Vert \nabla w\Vert^2_{L^2(\Omega)} \leq C \Vert  F \Vert_{L^2(\Omega)} \Vert w\Vert_{L^2(\Omega)} 
\end{equation*}
and then
\begin{equation}\label{inegaliteenergiebbiscompbis}
\vert \lambda\vert \Vert w\Vert_{L^2(\Omega)} + \vert \lambda\vert^{1/2} \Vert \nabla w\Vert_{L^2(\Omega)} \leq C \Vert  F \Vert_{L^2(\Omega)}.
\end{equation}
Writing the above equality as follows
$$
\lambda^2 \Vert w\Vert^2_{L^2(\Omega)} = -  \Vert \nabla w\Vert^2_{L^2(\Omega)} + \int_{\Omega} F\overline{w},
$$
and taking the module of the both sides of the previous equality, we get thanks to \eqref{inegaliteenergiebbiscompbis} the following estimate 
\begin{equation*}\label{estiminter3/23/4acomp4}
 \vert \lambda\vert^{3/2} \Vert w\Vert_{L^2(\Omega)} + \vert \lambda\vert^{3/4} \Vert \nabla w\Vert_{L^2(\Omega)} \leq C \Vert  F \Vert_{L^2(\Omega)}.
\end{equation*}
Repeating this reasoning, we find at Step $k$, with $k\geq 2$, the following estimate
\begin{equation*}\label{estiminterstepk3}
 \vert \lambda\vert^{r_k} \Vert w\Vert_{L^2(\Omega)} + \vert \lambda\vert^{\frac{1}{2}r_k} \Vert \nabla w\Vert_{L^2(\Omega)} \leq C \Vert  F \Vert_{L^2(\Omega)}
\end{equation*}
with 
$$
r_k = 1 +  \frac{1}{2}r_{k-1}, \; k \ge 1\quad\mathrm{ and  } \quad  r_0 = 0.
$$
We verify easily that for any $k \geq 1$, we have $r_k < 2$ and $\ds \lim_{k\rightarrow \infty}r_{k} = 2.$ We deduce that for any $r < 1$, there exists a constant $C_r$ depending only on $\Omega$, $\omega$ and $r$ such that we have the following estimate
\begin{equation}\label{estim2bbisrincomp}
 \vert \lambda\vert^{2r} \Vert w\Vert_{L^2(\Omega)} + \vert \lambda\vert^{r} \Vert \nabla w\Vert_{L^2(\Omega)} \leq C_r \Vert  F \Vert_{L^2(\Omega)}.
\end{equation} 
As $\Delta w = \lambda^2 w - F$, we get from \eqref{estim2bbisrincomp}
\begin{equation*}\label{estim2bbisDeltarincompvar}
  \vert \lambda\vert^{2(r-1)} \Vert \Delta w\Vert_{L^2(\Omega)} \leq C_r \Vert  F \Vert_{L^2(\Omega)}.
\end{equation*} 
From Rellich equality \eqref{Rellich}, we have
\begin{equation*}
\left\|\frac{\partial w}{\partial\textit{\textbf n}}\right\|^2_{L^2(\Gamma)} \leq C\Vert \nabla w\Vert^2_{L^2(\Omega)} + \Vert \nabla w\Vert_{L^2(\Omega)}\Vert \Delta w\Vert_{L^2(\Omega)}.
\end{equation*} 
So that by using \eqref{estim2bbisrincomp} and \eqref{estim2bbisDeltarincompvar} we get easily the following estimate
\begin{equation*}\label{estidernormFL2r}
\vert \lambda \vert^{\frac{3}{2}r-1} \, \left\|\frac{\partial w}{\partial\textit{\textbf n}} \right\|_{L^2(\Gamma)}   \leq C \Vert  F\Vert_{L^2(\Omega)}.
\end{equation*}
Finally thanks to \eqref{estim1bbisr} and \eqref{estimabdb} we deduce the required estimate \eqref{estim1BisCompD}.
\end{proof}
%

\begin{theorem}\label{th3TerD} $\mathrm{(}${\bf Resolvent Interior Estimates II for} $(\mathcal{P}_D^0) \mathrm{)}$ Let $\sqrt \rho F \in L^2(\Omega)$ and suppose $\lambda \in\mathbb{R}^* $ or $\lambda \in\mathbb{C} $ with  $\mathrm{Re}\, \lambda \ge \omega $ for some arbitrary $\omega > 0$. Then Problem $(\mathcal{P}_D^0)$ has a unique solution $w\in H^{3/2}_0(\Omega)$, with $\sqrt \rho  \nabla^2 w\in L^2(\Omega), \frac{\partial w}{\partial\textit{\textbf n}}\in L^2(\Gamma)$. \\
i) This solution satisfies the following estimate
\begin{equation*}\label{estim1TerD}
\vert \lambda \vert^{3/2} \,  \Vert w\Vert_{L^2(\Omega)}  +\;   \vert \lambda \vert^{1/2}  \Vert  w\Vert_{H^1(\Omega)} + \Vert w\Vert_{H^{3/2}(\Omega)} \; +  \;  \Vert \sqrt \rho  \nabla^2 w\Vert_{L^2(\Omega)}   \leq C \Vert \sqrt \rho F\Vert_{{L^2(\Omega)}}.
\end{equation*}
if $\lambda$ is real. If in addition $F\in H^{-1/2}(\Omega)$, then
\begin{equation}\label{estimDerivNormL2}
\left\|\frac{\partial w}{\partial\textit{\textbf n}}\right\|_{L^2(\Gamma)} \leq C \left( \Vert \sqrt \rho F\Vert_{{L^2(\Omega)}} + \Vert F\Vert_{{H^{-1/2}(\Omega)}} \right).
\end{equation} 
ii) In the complex case, for any $r < 1$, there exists a constant $C_r$ depending only on $r$, $\Omega$ and $\omega$ such that we have following inequality
\begin{equation*}\label{estim1TerE}
\begin{array}{rl}
\vert \lambda \vert^{\frac{3}{2}r} \,  \Vert w\Vert_{L^2(\Omega)} & +\;   \vert \lambda \vert^{\frac{1}{2}r}  \Vert  w\Vert_{H^1(\Omega)} +  \\\\
&   +  \;  \vert \lambda \vert^{r-1}( \Vert w\Vert_{H^{3/2}(\Omega)} + \Vert \sqrt \rho  \nabla^2 w\Vert_{L^2(\Omega)})     \leq C_r \Vert \sqrt \rho F\Vert_{{L^2(\Omega)}}.
\end{array}
\end{equation*}
 If in addition $F\in H^{-1/2}(\Omega)$, then
 \begin{equation*}\label{estimDerivNormL2bis}
\left\|\frac{\partial w}{\partial\textit{\textbf n}}\right\|_{L^2(\Gamma)} \leq C\vert \lambda \vert^{1-r}( \Vert \sqrt \rho F\Vert_{{L^2(\Omega)}} + \Vert F\Vert_{{H^{-1/2}(\Omega)}}).
\end{equation*} 

\end{theorem}

\begin{proof} For the existence and the uniqueness of the solution  $w\in H^{3/2}_0(\Omega)$, with $\sqrt \rho  \nabla^2 w\in L^2(\Omega)$, see \cite{AMN}. 

\noindent{\bf i) Real case.} Concerning the estimates, after multiplying by $w$, we get
\begin{equation}\label{1}
\displaystyle \vert \lambda \vert^{2} \,  \Vert w\Vert^2_{L^2(\Omega)} + \Vert \nabla w\Vert^2_{L^2(\Omega)} \leq \Vert \sqrt \rho F\Vert_{{L^2 (\Omega)}} 
\left\| \frac{w}{\sqrt \rho}\right\|_{L^2(\Omega)}.
\end{equation}
But, thanks to Hardy's inequality, we have
\begin{equation}\label{2}
\displaystyle \left\| \frac{w}{\sqrt \rho}\right\|_{L^2(\Omega)} \leq \Vert \frac{w}{\rho}\Vert^{1/2}_{L^2(\Omega)}\Vert w\Vert^{1/2}_{L^2(\Omega)}\leq C(\Omega) \Vert \nabla w\Vert^{1/2}_{L^2(\Omega)} \Vert w\Vert^{1/2}_{L^2(\Omega)}.
\end{equation}
So, using Young's inequality we firstly deduce that
\begin{equation}\label{3}
\vert \lambda \vert^{3/2} \,  \Vert w\Vert_{L^2(\Omega)}  \leq C(\Omega) \Vert \sqrt \rho F\Vert_{{L^2(\Omega)}}.
\end{equation}
Secondly, from \eqref{1}-{\eqref3}, we then obtain the following estimate
\begin{equation}\label{4}
\vert \lambda \vert^{1/2} \,  \Vert \nabla w\Vert_{L^2(\Omega)} \leq C(\Omega) \Vert \sqrt \rho F\Vert_{{L^2(\Omega)}}.
\end{equation}
Multiplying now by $-\rho\Delta w$ the equation $\lambda^2 w - \Delta w = F$, we get after integration by parts the following relation
$$
\displaystyle  \lambda^{2} \,  \Vert \sqrt \rho \nabla w\Vert^2_{L^2(\Omega)}  +  \Vert \sqrt \rho \Delta w\Vert^2_{L^2(\Omega)}  = \int_{\Omega} - F  \rho \Delta w -   \lambda^{2} \ \int_{\Omega} w \nabla \rho\cdot \nabla w.
$$
It follows by \eqref{3} and  \eqref{4} that
\begin{equation*}\label{4a}
\vert \lambda \vert \, \Vert \sqrt \rho  \nabla w\Vert_{L^2(\Omega)}  +  \Vert \sqrt \rho  \Delta w\Vert_{L^2(\Omega)}  \leq C(\Omega)\Vert \sqrt \rho F\Vert_{L^2(\Omega)}.
 \end{equation*}
As $F\in H^{-1}(\Omega)$, we see by Poincar\'e's inequality that
$$
\displaystyle  \lambda^{2} \,  \Vert w\Vert^2_{L^2(\Omega)} + \Vert \nabla w\Vert^2_{L^2(\Omega)}  = \langle F, w \rangle \leq C(\Omega)\Vert F\Vert_{{H^{-1} (\Omega)}} \Vert \nabla w\Vert_{L^2(\Omega)}
$$
which yields to
\begin{equation}\label{1a}
\Vert  w\Vert_{H^1(\Omega)} \leq C(\Omega)\Vert F\Vert_{{H^{-1} (\Omega)}}  \leq C(\Omega) \Vert \sqrt \rho F\Vert_{{L^2(\Omega)}}.
\end{equation}
Using \eqref{estimabe}, we deduce that
\begin{equation}\label{5}
\Vert w\Vert_{H^{3/2}(\Omega)} +  \Vert \sqrt \rho  \nabla^2 w\Vert_{L^2(\Omega)}  \leq C \Vert \sqrt \rho F\Vert_{L^2(\Omega)}.
\end{equation}
For the last estimate concerning the normal derivative of $w$, we start by use Rellich identity \eqref{Rellich}. It follows that
\begin{equation*}\label{6a}
\ds \left\|\frac{\partial w}{\partial\textit{\textbf n}}\right\|^2_{L^2(\Gamma)} \leq C \left( \ds \lambda^{2} \,  \Vert \nabla w\Vert_{L^2(\Omega)}  \Vert  w\Vert_{L^2(\Omega)}  + \Vert  \textbf{\textit h}\cdot \nabla w \Vert_{H^{1/2}(\Omega)} \Vert F\Vert_{H^{-1/2}(\Omega)} +  \Vert \nabla w\Vert^2_{L^2(\Omega)}\right).
\end{equation*}
The estimate \eqref{estimDerivNormL2} is then a consequence of the regularity $\mathcal{C}^\infty(\overline{\Omega})$ of the vector field $\textbf{\textit h}$ and the estimates \eqref{5}, \eqref{1a}, \eqref{4} and \eqref{3}.

\noindent{\bf ii) Case $\lambda$ complex.} Multiplying by $\overline{w}$ and using \eqref{2}, we get the following inequality
\begin{equation*}
\vert \lambda\vert \Vert w\Vert^2_{L^2(\Omega)} +\Vert \nabla w\Vert^2_{L^2(\Omega)} \leq C_\omega \Vert  \sqrt \rho\, F \Vert_{L^2(\Omega)} \Vert \frac{w}{ \sqrt \rho}\Vert_{L^2(\Omega)} \leq  C_\omega \Vert  \sqrt \rho\, F \Vert_{L^2(\Omega)} \Vert\Vert \nabla w\Vert^{1/2}_{L^2(\Omega)}\Vert w\Vert^{1/2}_{L^2(\Omega)}.
\end{equation*}
So, using Young's inequality we deduce that
\begin{equation*}\label{31}
\vert \lambda\vert \Vert w\Vert^2_{L^2(\Omega)} +\Vert \nabla w\Vert^2_{L^2(\Omega)} \leq C_\omega  \Vert \sqrt \rho F\Vert^{4/3}_{{L^2(\Omega)}}\Vert w\Vert_{L^2(\Omega)}^{2/3}.
\end{equation*}
and then 
\begin{equation*}\label{inegaliteenergiebbiscompBis}
\vert \lambda\vert ^{3/4} \Vert w\Vert_{L^2(\Omega)} \leq C_\omega  \Vert \sqrt \rho F\Vert .
\end{equation*}
From this last estimate, it follows that
\begin{equation}\label{inegaliteenergiebbiscompTer}
\vert \lambda\vert ^{1/4}\Vert \nabla w\Vert_{L^2(\Omega)} \leq C_\omega  \Vert \sqrt \rho F\Vert .
\end{equation}
In particular, we have also
\begin{equation*}\label{inegaliteenergiebbiscompQuart}
\vert \lambda\vert^{1/2} \left\|\frac{w}{\sqrt \rho}\right\|_{L^2(\Omega)} \leq C_\omega  \Vert \sqrt \rho F\Vert_{L^2(\Omega)} .
\end{equation*}
We then proceed as in the proof of Theorem \ref{th3TerD}. We know that
$$
\lambda^2 \Vert w\Vert^2_{L^2(\Omega)} = -  \Vert \nabla w\Vert^2_{L^2(\Omega)} + \int_{\Omega} F\overline{w}.
$$
So taking the module of the both sides of the previous equality, we get thanks to \eqref{inegaliteenergiebbiscompTer} and \eqref{inegaliteenergiebbiscompQuart} the following estimate 
\begin{equation}\label{estiminter3/23/4acomp0}
 \vert \lambda\vert^{5/4} \Vert w\Vert_{L^2(\Omega)} + \vert \lambda\vert^{5/12} \Vert \nabla w\Vert_{L^2(\Omega)} \leq C \Vert \sqrt \rho  F \Vert_{L^2(\Omega)}.
\end{equation}
Repeating this reasoning, we find  the following estimate: for any $r < 1$
\begin{equation*}\label{estiminterstepBis}
 \vert \lambda\vert^{\frac{3}{2}r} \Vert w\Vert_{L^2(\Omega)} + \vert \lambda\vert^{\frac{1}{2}r} \Vert \nabla w\Vert_{L^2(\Omega)} +  \vert \lambda\vert ^{r} \left\| \frac{w}{\sqrt \rho}\right\|_{L^2(\Omega)} \leq C_r \Vert \sqrt \rho  F \Vert_{L^2(\Omega)}.
\end{equation*}
Multiplying by $-d\Delta \overline{w}$ the equation $\lambda^2 w - \Delta w = F$ and using the above estimates we get
\begin{equation*}\label{4abc}
\vert \lambda \vert \, \Vert \sqrt \rho  \nabla w\Vert^2_{L^2(\Omega)}  +  \Vert \sqrt \rho  \Delta w\Vert^2_{L^2(\Omega)}  \leq C( \Vert \sqrt \rho F\Vert_{L^2(\Omega)}\Vert \sqrt \rho  \Delta w\Vert_{L^2(\Omega)} +  \vert \lambda \vert^2 \Vert w\Vert_{L^2(\Omega)}\Vert \nabla w\Vert_{L^2(\Omega)}).
 \end{equation*}
From this last inequality,  \eqref{estiminter3/23/4acomp0} and \eqref{estimabe}, we deduce that
\begin{equation*}\label{5aa}
\Vert w\Vert_{H^{3/2}(\Omega)} +  \Vert \sqrt \rho  \nabla^2 w\Vert_{L^2(\Omega)}  \leq C_r \vert \lambda\vert^{1-r}\Vert \sqrt \rho F\Vert_{L^2(\Omega)}.
\end{equation*}
Proceeding as in the real case, it follows that 
\begin{equation*}\label{6ab}
\left\| \frac{\partial w}{\partial\textit{\textbf n}}\right\|_{L^2(\Gamma)}^2 \leq C_r \left(\vert \lambda\vert ^{2-2r}\, \Vert \sqrt \rho  F \Vert^2_{L^2(\Omega)} + \vert \lambda\vert ^{1-r} \Vert \sqrt \rho  F \Vert_{L^2(\Omega)}\Vert F\Vert_{{H^{-1/2}(\Omega)}}\right),
\end{equation*}
\emph{i.e} 
\begin{equation*}\label{6abc}
\left\|\frac{\partial w}{\partial\textit{\textbf n}}\right\|_{L^2(\Gamma)} \leq C_r \left(\vert \lambda\vert ^{1-r}\ \Vert \sqrt \rho  F \Vert_{L^2(\Omega)} +  \Vert F\Vert_{{H^{-1/2}(\Omega)}}\right).
\end{equation*}
\end{proof}

\begin{theorem}\label{th4} $\mathrm{(}${\bf Resolvent Boundary Estimates for} $(\mathcal{P}_D^H))$ Let $g\in H^1(\Gamma)$ and suppose $\lambda \in\mathbb{C} $ with  $\mathrm{Re}\, \lambda \ge \omega$ for some arbitrary $\omega > 0$. Then Problem $(\mathcal{P}_D^H)$ has a unique solution which satisfies 
\begin{equation}\label{regulariteD}
u\in H^{3/2}(\Omega),\quad \sqrt {\rho} \, \nabla^2 u \in {\textbf{\textit L}}^2(\Omega) \quad and \quad \frac{\partial u}{\partial \textbf{\textit n}}\in L^2(\Gamma).
\end{equation}
\end{theorem}

\begin{theorem}\label{th31D} $\mathrm{(}${\bf Resolvent Boundary Estimates for $(\mathcal{P}_D^H)$} Let $g\in H^1(\Gamma)$ and suppose $\lambda \in\mathbb{C} $ with  $\mathrm{Re}\, \lambda \ge \omega $ for some arbitrary $\omega > 0$. Then for any real number $ r < 1$, there exists a positive constant $C_r$, depending only on $r$, $\Omega$ and $\omega$, such that the solution given by the previous theorem satisfies the following estimate:
\begin{equation*}\label{estim1bbisrD}
\begin{array}{rl}
\vert \lambda \vert^{\frac{3}{2}r} \,  \Vert u\Vert_{L^2(\Omega)}  &+\;   \vert \lambda \vert^{\frac{1}{2}r}  \Vert  u\Vert_{H^1(\Omega)}  + \;  \Vert  \frac{\partial u}{\partial \textbf{\textit n}}\Vert_{L^2(\Gamma)} +\\\\
& +  \;  \vert \lambda \vert^{r-1}( \Vert u\Vert_{H^{3/2}(\Omega)} + \Vert \sqrt \rho  \nabla^2 u\Vert_{L^2(\Omega)})  \leq  C_r (\vert \lambda \vert \, \Vert  g\Vert_{L^2(\Gamma)} + \Vert g\Vert_{H^1(\Gamma)}).
\end{array}
\end{equation*}
\end{theorem}

\begin{proof}  Clearly we have the following relation:
\begin{equation*}\label{egaliteenergiebbis2}
\lambda^2 \Vert u\Vert^2_{L^2(\Omega)} + \Vert \nabla u\Vert^2_{L^2(\Omega)} = \int_{\Gamma} g\frac{\partial \overline{u}}{\partial \textbf{\textit n}}
\end{equation*}
and the estimate
\begin{equation}\label{inegaliteenergiebbis3}
\vert\lambda\vert \Vert u\Vert^2_{L^2(\Omega)} + \Vert \nabla u\Vert^2_{L^2(\Omega)} \leq C\Vert g\Vert_{L^2(\Gamma)}\left\|\frac{\partial \overline{u}}{\partial \textbf{\textit n}}\right\|_{L^2(\Gamma)}.
\end{equation}
Using Rellich identity \eqref{RellichBi} we have
\begin{equation*}\label{RelavecL2u}
\ds \left\|\frac{\partial u}{\partial \textbf{\textit n}}\right\|_{L^2(\Gamma)}^2 \leq C \left( \ds \int_\Gamma (\vert\lambda\vert^2 \vert g \vert^2  + \left|\nabla_{\mathscr{T}}g\right|^2) + \Vert \nabla u\Vert^2_{L^2(\Omega)} + \vert\lambda\vert^2 \Vert u \Vert^2_{L^2(\Omega)}\right).
\end{equation*}
From this last inequality and \eqref{inegaliteenergiebbis3} we get 
\begin{equation*}\label{RellichDbis}
\left\| \frac{\partial u}{\partial \textbf{\textit n}}\right\|_{L^2(\Gamma)} \leq C(\Omega)\Big(\vert \lambda \vert\Vert g \Vert_{L^2(\Gamma)} + \Vert g\Vert_{H^1(\Gamma)}  \Big),
\end{equation*}
where the constant $C(\Omega)$ depends only on the Lipschitz character of $\Omega$ and $\omega$. Using then the estimates \eqref{estim1bbisr} and \eqref{estimabdb}, we get the inequality  \eqref{estim1bbisrD}. 
\end{proof}

We consider now the following Dirichlet-to-Neumann operator : let $g\in H^1(\Gamma)$ and $u_g\in H^{3/2}(\Omega)$ the unique solution satisfying
$$
 \lambda^2 u_g - \Delta u_g = 0 \quad \mathrm{in}\; \Omega \quad \mathrm{and}\quad u_g = g \quad \mathrm{on}\; \Gamma .
$$
Let us define the operator
$$
S_\lambda : g\mapsto  \frac{\partial u_g}{\partial \textbf{\textit n}}
$$ 
and provide the space $H^1(\Gamma)$ with the following norm equivalent to the usual norm of $H^1(\Gamma)$: for $g\in H^1(\Gamma)$, we set  
$$
\vert \vert\vert g \vert\vert\vert = \vert \lambda \vert \Vert g\Vert_{L^2(\Gamma)} + \Vert g \Vert_{H^1(\Gamma)}.
$$
We denote by $H^1_\lambda(\Gamma)$ the space $H^1(\Gamma)$ equipped with this norm and by $H^{-1}_\lambda(\Gamma)$ its dual space. We know that
\begin{equation*}\label{d07-150318-e1B}
S_\lambda :\  H^1_\lambda (\Gamma)\  \longrightarrow\  L^2(\Gamma)  
\end{equation*}
is continuous and there exists a constant C, depending only on $\Omega$ if $\lambda$ is real and on $\Omega $ and $ \omega$ if $\lambda$  is a complex number with $Re\,  \lambda \ge \omega > 0$, such that
$$
\left\|  \frac{\partial u_g}{\partial \textbf{\textit n}} \right\|_{L^2(\Gamma)} \leq C\vert \vert\vert g \vert\vert\vert .
$$
We see easily that for any $f\in H^1_\lambda (\Gamma)$ and any $g\in L^2(\Gamma)$ we have
$$
\langle  \frac{\partial u_g}{\partial \textbf{\textit n}} , f \rangle_{H^{-1}_\lambda(\Gamma)\times H^{1}_\lambda(\Gamma)} = \langle  g , \frac{\partial u_f}{\partial \textbf{\textit n}}\rangle_{L^2(\Gamma)\times L^2(\Gamma)}.
$$
So the operator $S_\lambda$ is self-similar and thanks to  Theorem \ref{th2} and Theorem \ref{th31D} is continuous from $ L^2(\Gamma)$ into  $H^{-1}_\lambda(\Gamma)$.

Finally, proceeding as in the proof of Theorem \ref{th2Bi}, we can obtain the following theorem.


\begin{theorem}\label{th2Bii} $\mathrm{(}${\bf Very Weak Estimate for} $(\mathcal{P}_D^H) \mathrm{)}$ Let $g\in L^{2}(\Gamma)$ and suppose $\lambda \in\mathbb{R}^* $ or  $\lambda \in\mathbb{C} $ with  $\mathrm{Re}\, \lambda \ge \omega $ for some arbitrary $\omega > 0$. Then Problem $(\mathcal{P}_D^H)$ has a unique solution which satisfies 
\begin{equation*}
u\in H^{1/2}(\Omega),\quad    \frac{\partial u}{\partial\textit{\textbf n}}  \in H^{-1}(\Gamma)\
\end{equation*}
with the following estimate:
\begin{equation*}\label{estim9bis}
     \vert \lambda \vert^{1/2} \Vert u\Vert_{L^2(\Omega)} +  \Vert u\Vert_{H^{1/2}(\Omega)} +  \left\| \frac{\partial u}{\partial\textit{\textbf n}} \right\|_{H^{-1}_\lambda(\Gamma)} \leq C \left\|  g\right\|_{L^2(\Gamma)},
\end{equation*}
where $C$ is not depending on $\lambda$.
\end{theorem}  

\section*{Acknowledgements}
We would like to thank the referee for his (her) valuable comments which enabled us to
improve substantially the paper.

\end{document}